\documentclass[12pt, oneside,usenames,dvipsnames]{amsart}

\usepackage{amsmath,ifthen, amsfonts, amssymb, yhmath,
srcltx, % COMMENT OUT
   amsopn,tikz, tkz-euclide, url}
\usetikzlibrary{decorations.markings,arrows, cd}
\usepackage{tikz-cd}
 \usepackage{dutchcal}
\usepackage{xfrac} 
\usepackage{graphicx, ifsym}
\usepackage{fullpage}
\tikzset{>=latex}

\usepackage{thmtools}
\usepackage{thm-restate}
\usepackage{pgfplots}
\usetikzlibrary{intersections, pgfplots.fillbetween}

\usepackage{hyperref}

\usepackage{cleveref}

\usepackage{soul} % use this (many fancier options)
\setstcolor{red}

\newcommand{\showcomments}{yes}
\renewcommand{\showcomments}{no}

\newsavebox{\commentbox}
\newenvironment{com}%
{\ifthenelse{\equal{\showcomments}{yes}}%
{\footnotemark
        \begin{lrbox}{\commentbox}
        \begin{minipage}[t]{1.25in}\raggedright\sffamily\tiny
        \footnotemark[\arabic{footnote}]}
{\begin{lrbox}{\commentbox}}}%
{\ifthenelse{\equal{\showcomments}{yes}}%
{\end{minipage}\end{lrbox}\marginpar{\usebox{\commentbox}}}
{\end{lrbox}}}

\usepackage{amsmath,amsfonts, amssymb, graphicx, tikz, bbding,ifthen,srcltx,amsopn, url}
\usepackage{bm}
\usetikzlibrary{decorations.markings,arrows, intersections}
\usetikzlibrary{decorations.pathreplacing}

\DeclareMathOperator{\Stab}{Stab}

\theoremstyle{definition}
\newtheorem{thm}{Theorem}[section]
\newtheorem{lem}[thm]{Lemma}
\newtheorem{prop}[thm]{Proposition}

\newtheorem{question}[thm]{Question}
\newtheorem{remark}[thm]{Remark}
\newtheorem{defn}[thm]{Definition}
\newtheorem{cor}[thm]{Corollary}

\newtheorem{notn}[thm]{Notation}

\author{Kasia Jankiewicz}
\address{Department of Mathematics, University of British Columbia, Canada}
\email{kasia@math.ubc.ca}
\title{Finite Stature in Artin groups}

\begin{document}
\begin{com}
{\bf \normalsize COMMENTS\\}
ARE\\
SHOWING!\\
\end{com}

\begin{abstract} We give criteria for a graph of groups to have finite stature with respect to its collection of vertex groups, in the sense of Huang-Wise. We apply it to the triangle Artin groups that were previously shown to split as a graph of groups. This allows us to deduce residual finiteness, and expands the list of Artin groups known to be residually finite.
\end{abstract}

%\date{\today}

%%%
\maketitle

%\tableofcontents
\section{Introduction}
A group $G$ has \emph{finite stature} with respect to a collection of subgroups $\Omega$, if for every $H\in \Omega$ there are only finitely many $H$-conjugacy classes of subgroups of the form $H\cap\bigcap_{i\in I}H_i^{g_i}$ where $H_i^{g_i}$ is a $G$-conjugate of an element $H_i\in\Omega$. Finite stature was introduced by Huang-Wise in \cite{HuangWiseStature19} where they proved that under certain assumptions the fundamental group $G$ of a graph of groups has certain separability properties, provided that $G$ has finite stature with respect to its collection of vertex groups. In \cite{HuangWiseSpecial19} the same authors showed that a graph of nonpositively curved cube complexes $X$ with word hyperbolic fundamental group is virtually special, provided that $\pi_1X$ has finite stature with respect to the vertex groups in the corresponding splitting as a graph of groups. Finite stature is closely related to the more classical notion of finite \emph{height}, introduced and studied in \cite{GMRS98}.

The goals of this paper are two-fold. Firstly, we illustrate that the notion of finite stature is satisfied and useful in well-studied groups arising naturally in topology. Indeed, we provide explicit examples of very different nature than the groups studied in \cite{HuangWiseStature19, HuangWiseSpecial19}, as they are not hyperbolic and not virtually compact special. Specifically we show that the splittings of certain Artin groups obtained by the author in \cite{JankiewiczArtinRf, JankiewiczArtinSplittings} have finite stature with respect to the vertex groups. Secondly, we deduce the residual finiteness of those Artin groups, which was previously not known in some cases.

A \emph{triangle Artin group} is an Artin group on three generators, given by the presentation
$$G_{MNP} = \langle a, b, c\mid (a, b)_{M} = (b, a)_{M}, \,\,\,(b,c)_{N} = (c,b)_{N}, \,\,\,(c,a)_{P} = (a,c)_{P}\rangle,$$
where $(a,b)_M$ denotes the alternating word $aba\dots$ of length $M$. 

\begin{thm}\label{thm: main}
A triangle Artin group $G_{MNP}$ splits as graphs of free groups with finite stature with respect to its collection of vertex groups, provided that either $M>2$ or $N>3$, where we assume that $M\leq N\leq P$.
\end{thm}
As a consequence (using results of \cite{HuangWiseStature19}) we obtain the following.
\begin{cor}
A triangle Artin group $G_{MNP}$, where $M\leq N\leq P$ and either $M>2$ or $N>3$, is residually finite.
\end{cor}

The condition on $M,N,P$ in Theorem~\ref{thm: main} excludes the cases  $(M,N,P)=(2,2,P)$ and $(M,N,P)=(2,3,P)$. In the first case, the corresponding Artin group $G_{MNP}$ is isomorphic to $Z\times A_P$ where $A_P$ denotes a dihedral Artin group, and consequently $G_{MNP}$ does not split as a graph of free groups, but is well-known to be residually finite.
Wu-Ye have recently shown that $G_{2,3,P}$ with $p\geq 6$ splits as a graphs of finite rank free groups if and only if $P$ is even \cite{wu2023splittings}. In a subsequent work, Meyer proved that $G_{23P}$ with $P$ even has finite stature with respect to its collection of vertex groups \cite{Meyer24}. It remains open whether $G_{23P}$ with $P\geq 7$ is residually finite.

There are a few other classes of Artin group that are known to be residually finite. 
 In the case of spherical type Artin groups, residual finiteness follows from linearity \cite{Krammer2002,Bigelow2001,CohenWales2002,Digne2003}. 
The linearity of a few other Artin groups was established as a consequence of being virtually special \cite{LiuGraphManifolds, PrzytyckiWiseGraphManifolds}, but none of the triangle Artin groups considered in Thereom~\ref{thm: main} admit virtual geometric actions on CAT(0) cube complexes \cite{HuangJankiewiczPrzytycki16, HaettelArtin}. Residual finiteness of some other Artin groups was proven in \cite{BlascoGarciaJuhaszParis18, BlascoGarciaMartinezPerezParis19}. 

Some, but not all, of the groups considered in the above corollary were proven to virtually split as \emph{algebraically clean} graphs of free groups, i.e.\ graphs of finite rank free groups where all inclusions of edge groups in the adjacent vertex groups are inclusions as free factors, in \cite{JankiewiczArtinRf, JankiewiczArtinSplittings}. Such groups are known to be residually finite \cite{WisePolygons}.
Our method allows us to deduce residual finiteness of new Artin groups, but also recover the residual finiteness of the Artin groups treated in \cite{JankiewiczArtinRf, JankiewiczArtinSplittings}.

Group virtually splitting as algebraically clean graphs of groups satisfy some stronger profinite properties than residual finiteness, 
some of which are discussed in the forthcoming paper \cite{JankiewiczSchreveProfiniteAlgebraicallyClean}. 
We do not know whether all the groups considered in this paper are in fact virtually algebraically clean. More generally, the following is open.
\begin{question}
Let $G$ be a graph of finite rank free groups with finite stature with respect to its collection of vertex groups. Does $G$ have a finite index subgroup whose induced splitting is algebraically clean?
\end{question}
The converse is known to be false, as there are examples of algebraically clean graphs of free groups that do not have finite stature \cite[Exmp 3.31]{HuangWiseStature19}. On the other hand, we do not know whether there exists a group $G$ splitting as an algebraically clean graph of groups such that $G$ does not have finite stature with respect to \emph{any} splitting with free vertex groups.

This paper is organized as follows. In Section~\ref{sec: maps} we state some facts about maps between graphs and free groups, and fix the notation and terminology. Section~\ref{sec: finite stature} discusses the notion of finite stature, and we prove some facts used later in the text. Section~\ref{sec: graphs of groups} studies certain families of graphs of free groups. Finally, Section~\ref{sec: Artin} is devoted to Artin groups, and contains computations that allow us to apply the results from earlier sections to prove Theorem \ref{thm: main}.

\subsection*{Acknowledgements}
The author thanks Jingyin Huang and Dani Wise for fruitful discussions, and the anonymous referee for their corrections and suggestions. This material is based upon work supported by the National Science Foundation grants DMS-2203307,  DMS-2238198, and DMS-1926686.  

\section{Preliminaries}\label{sec: maps}
\subsection{Maps between graphs}
A \emph{combinatorial graph} $\Gamma$ is a disjoint union $V(\Gamma)\sqcup E(\Gamma)$ together with the operation $E(\Gamma)\to E(\Gamma), e\mapsto \bar e$ of taking the \emph{opposite edge} (i.e.\ the same edge with opposite orientation), and the operation $E(\Gamma)\to V(\Gamma), e\mapsto \tau(e)$ of taking the \emph{endpoint} of an oriented edge.

A \emph{metric graph} is a combinatorial graph that can also be viewed as a $1$-dimensional CW-complex, with a path metric in which each $1$-cell has length $1$. 
Later, we will be considering graphs of free groups and corresponding graphs of spaces where the spaces are graphs as well. We will denote the underlying graph of the graph of groups/graphs by $\Gamma$, while the vertex and edge spaces will be denoted by letters such as $X, Y$ and will be viewed as metric graphs. The following definitions will be applied to graphs arising as vertex and edge spaces.

A continuous map $\phi:Y\to X$ between two metric graphs is \emph{combinatorial}, if the image of each $0$-cell of $Y$ is a $0$-cell of $X$, and while
restricted to an open $1$-cell with endpoints $y_1,y_2$, $\phi$ is an isometry onto an edge with endpoints $\phi(y_1), \phi(y_2)$. A \emph{combinatorial immersion} is a combinatorial map $f:Y\to X$ which is locally injective. Every combinatorial immersion is $\pi_1$-injective \cite[Prop 5.3]{Stallings83}. Also, every combinatorial immersion can be ``completed'' to a covering map by attaching trees to $Y$, without changing its homotopy type.

A \emph{Stalling's fold} is a combinatorial map $f:Y\to X$ where 
\begin{itemize}
\item there exist distinct edges $y_1, y_2\in E(Y)$ such that $\tau(\bar y_1) = \tau(\bar y_2)$, and $E(X) = E(Y)/ {y_1\sim y_2}$,
\item $V(X) = V(Y)/ \tau(y_1) \sim \tau(y_2)$, and
\item $f$ is the natural quotient map, where $f(y_1) = f(y_2)$.
\end{itemize}
We note that $f$ is a homotopy equivalence if and only if $ \tau(y_1) \neq \tau(y_2)$.

We will also consider more general maps between graphs than combinatorial.
\begin{defn}
A continuous map $\phi:Y\to X$ between two metric graphs is \emph{monotone}, if the image of each $0$-cell of $Y$ is a $0$-cell of $X$, and while restricted to each 1-cell $y$ of $Y$, $\phi$ is either constant and its image is a $0$-cell $x$ in $X$, or $\phi$ is a combinatorial map after possibly subdividing $y$ into $n$ nontrivial subintervals.
 \end{defn}
 
 Here are two important examples of monotone maps.
 An \emph{edge-subdivision} is a monotone map $f:Y\to X$ where 
\begin{itemize}
\item there exists an edge $y\in E(Y)$ and edges $y_1, \dots, y_k \in E(X)$ where $k\geq 2$ such that $f(y)$ is equal the path $y_1\cdots y_k$,
\item $E(Y)-\{y\} = E(X) -\{y_1, \dots, y_k \}$, and $f$ is the identity map on $E(Y)-\{y\}$,
\item $V(X) = V(Y) \sqcup \{\tau(y_1), \dots, \tau(y_{k-1})\}$, and $\tau(y_i) = \tau(\bar y_{i+1})$ for all $i=1, \dots, k-1$ (i.e. $y_1\cdots y_k$ is a path in $X$),
\end{itemize}
An edge-subdivision is always a homotopy equivalence.
An \emph{edge-collapse} is a monotone map $f:Y\to X$ where 
\begin{itemize}
\item there exist an edge $y \in E(Y)$ such that $E(X) = E(Y)-\{y\}$,
\item $V(X) = V(Y)/ \tau(y) \sim \tau(\bar y)$, and
\item $f$ is the natural quotient map, where $f_{|y}$ is constant.
\end{itemize}
Similarly, an edge-collapse is a homotopy equivalence if and only if $ \tau(y) \neq \tau(\bar y)$, i.e.\ if $y$ is not a loop.

The following proposition provides a useful factorization of every monotone map.
\begin{prop}\label{prop:factorization}
Every monotone map $\phi:Y\to X$ factors as $Y\xrightarrow{\sigma} \overline Y\xrightarrow{\iota} X$ where
\begin{itemize}
\item $\sigma:Y\to \overline Y$ is obtained by a sequence of edge-subdivisions, Stalling's folds, edge-collapses,
\item $\iota: \overline Y\to X$ is a combinatorial immersion.
\end{itemize}
\end{prop}

\begin{proof} Every combinatorial map factors as a sequence of Stallings-folds followed by a combinatorial immersion \cite[Sec 3.3]{Stallings83}. By definition, a monotone map $\phi$ restricted to an edge is either an edge-collapse, or an edge-subdivision post-composed with a combinatorial map. The statement follows.
\end{proof}

\subsection{Subgroups of free groups}
Let $X$ be a metric graph with a basepoint $x\in X$, and let $F=\pi_1(X,x)$.

As noted above, for every combinatorial immersion $Y  \to X$ and $y\in Y$ that maps to $x$, the fundamental group $\pi_1(Y,y)$ naturally embeds as a subgroup of $F$ \cite[Prop 5.3]{Stallings83}, in which case we say \emph{$(Y,y)\to (X,x)$ represents $\pi_1(Y,y)$}. Two such groups $\pi_1(Y,y)$ and $\pi_1(Y,y')$ are conjugate in $F$ by an element of $F$ represented by the loop in $X$ that is an image of a path from $y$ to $y'$ in $Y$. More generally, any other $F$-conjugate of $\pi_1(Y,y)$, by say an element $g\in F$,   can be obtained by taking a union of $Y$ and a path labelled by $g$ with its start-point attached to $y$ and performing Stalling folds. The fundamental group of this new graph based at the endpoint of the added path is the conjugate $g^{-1}\pi_1(Y,y)g$.  Thus a map $Y\to X$ without specifying the basepoint in $Y$ determines an $F$-conjugacy class of subgroups of $F$. 
 
 Also, for every combinatorial immersion $Y  \to X$, there exists a covering map $\hat X\to X$ such that $Y \to X$ factors as an embedding $Y\to \hat X$ which is a homotopy equivalence, composed with the covering map $\hat X\to X$ \cite[Thm 6.1]{Stallings83}.

\begin{defn}[{\cite[Sec 7]{Stallings83}}]
Given a subgroup $H\subseteq F$, the \emph{core of $H$ with respect to $X$} is a based combinatorial immersion $i:(Y , y)\to (X, x)$ where $Y$ is the minimal subgraph of the covering space of $\hat X\to X$ that corresponds to $H$, where the inclusion $Y\to \hat X$ is a homotopy equivalence, and in particular induces an isomorphism $\pi_1(Y,y) \to \pi_1(\hat X, y) = H$.\end{defn}

We can also think of the core of an $F$-conjugacy class of $H$ as a combinatorial immersion $Y\to X$ where $Y$ is minimal among cores $(Y,y) \to (X,x)$ of subgroups in the conjugacy class. A core of an $F$-conjugacy class has no leaves, as removing leaves does not affect homotopy type of a graph.

When $(Y,y)\to (X,x)$ is a monotone map and $Y$ has no leaves (except for $y$ possibly), then $\iota:(\overline Y, \sigma(y))\to (X, x)$ in Proposition~\ref{prop:factorization} is the core of $\pi_1(Y,y)\subseteq \pi_1(X,x)$.

\subsection{Intersections of subgroups}

Let $\phi_i: Y_i\to X$ be a combinatorial immersion for $i=1,2$.
The \emph{fiber product of $Y_1$ and $Y_2$ over $X$} is the graph $Y_1\otimes_X Y_2$ with the vertex set
\[
\{(y_1,y_2)\in V(Y_1)\times V(Y_2): \phi_1(y_1) = \phi_2(y_2)\}
\] 
and the edge set
\[
\{(e_1, e_2)\in E(Y_1)\times E(Y_2): \phi_1(e_1) = \phi_2(e_2)\}.
\]
There is a natural combinatorial immersion $Y_1\otimes_X Y_2 \to X$, given by $(y_1, y_2)\mapsto \phi_1(y_1) = \phi_2(y_2)$.  

\begin{lem}[{\cite{Stallings83}}]\label{lem:fiber product}
Let $H_1, H_2\subseteq G = \pi_1 (X, v)$ where $X$ is a finite metric graph, and for $i=1,2$ let $(Y_i, y_i)\to (X, x)$ be the core of $H_i$ with respect to $X$. Then the intersection $H_1\cap H_2$ is represented by $(Y_1 \otimes_{X} Y_2, (y_1, y_2)) \to (X,x)$.
\end{lem}

We emphasize that in general fiber products have multiple connected components. When $(Y_i, y_i)\to (X, x)$ represents $H_i$, then each connected component of $Y_1 \otimes_{X} Y_2$ represents a conjugacy class of a subgroup of $G$ of the form $H_1^{g_1}\cap H_2^{g_2} = g_1^{-1}H_1g_1\cap g_2^{-1}H_2g$ for some $g_1, g_2\in G$.

\begin{lem}[{\cite[Lem 1.2]{GMRS98}}] \label{lem: finitely many conjugacy classes of intersections}
 Suppose $G$ is a Gromov hyperbolic group, and $H_1, H_2$ are quasiconvex subgroups.
Then there are only finitely many conjugacy classes of infinite intersections $H_1^{g_1}\cap H_2^{g_2}$ where $g_1,g_2\in G$. 
\end{lem}

The above statement follows directly from \cite[Lem 1.2]{GMRS98} when $H_1 = H_2$. Their proof also works in the general case, and for completeness we include it below.

\begin{proof} 

By Lemma~2.4 each conjugacy class of the intersection of conjugates $H_1$ and $H_2$ is represented by  the connected component of the fiber product $Y_1\otimes_X Y_2$ where $Y_1, Y_2$ are cores of $H_1, H_2$ with respect to $X$. Since $H_1, H_2$ have finite ranks, $Y_1, Y_2$ are finite graphs. Thus $Y_1\otimes_X Y_2$ is finite, and in particular, $Y_1\otimes_X Y_2$ has finitely many connected components (each representing a conjugacy class of the intersections of conjugates of $H_1$ and $H_2$).

Following \cite[Lem 1.2]{GMRS98} we will show that if $H_1, H_2$ are $K$-quasiconvex subgroups of a $\delta$-hyperbolic group $G$, then for every shortest representative $g$ of the double coset $H_1gH_2$, if $|g|\geq 2k+2\delta$, then $H_1\cap H_2^g$ is finite. That will imply our statement. Indeed, each intersection $H_1^{g_1}\cap H_2^{g_2}$ is conjugate to $H_1\cap H_2^g$, and there are only finitely many $g\in G$ such that $|g|<2K+2\delta$, hence only finitely many conjugacy classes of infinite subgroups of the form $H_1^{g_1}\cap H_2^{g_2}$.

Fix a double coset $H_1gH_2$ and assume that every representative $g$ of this double coset has length $|g|\geq  2k+2\delta$.
Since the cardinality of a subgroup $H_1\cap g^{-1}H_2g$ is invariant under conjugation, we can assume that $g$ is the shortest representative of the coset $gH_1$. Let $h\in H_1\cap g^{-1}H_2g$, and let $h_0\in H_2$ satisfy $h = g^{-1}h_0g$.
Consider a quadrilateral in the Cayley graph of $G$ at points $1, g^{-1}, g^{-1}h_0, h=g^{-1}h_0g$ with geodesic path $p_h$ going from $1$ to $h$, $p_1$ going from $1$ to $g^{-1}$, $p_{h_0}$ going from $g^{-1}h_0$, and $p_2$ going from $g^{-1}h_0$ to $g^{-1}h_0g = h$.  We will denote the label of a path $p$ by $Lab(p)$, the standard distance in the Cayley graph by $d(\cdot, \cdot)$, and the length of a geodesic by $|\cdot|$, and a path $x$ with reversed orientation by $\overline x$.

Let $v$ be a vertex on the path $p_h$ which is as close to the middle as possible, in particular $|d(1, v)-d(v,h)|\leq 1$. Let $q$ be an initial subpath of $p_h$ going from $1$ to $v$. Since $H_1$ is $K$-quasiconvex and $h\in H_1$, there exists a vertex at distance at most $K$ away from $v$ which belongs to $H_1$, let $s$ be a path from $v$ to that vertex. Let $t$ be a shortest path from $v$ to $p_{h_0}$, and let $w$ denote its endpoint. Since $H_2$ is quasiconvex and $h_0\in H_2$,  there exists an element of the coset $g^{-1}H_2$ at most $K$ away from $w$, let $s'$ be a path from that vertex to $w$.

We have $g = Lab(q'\overline{s'})Lab(s'\overline{t}s)Lab(\overline{s}\overline{q})$, and we have $Lab(q'\overline{s'})\in H_2$, $Lab(s'\overline{t}s)\in H_2gH_1$, and $Lab(\overline{s}\overline{q})\in H_1$. By our assumption $|Lab(s'\overline{t}s)|> 2K+2\delta$, so $|t|\geq |Lab(s'\overline{t}s)| - |s| - |s'| > 2\delta$. Since every point of the quadrilateral is contained in $2\delta$-neighborhood of the other three sides, we conclude that $v\in N_{2\delta}(p_1\cup p_2)$. 

Suppose that $v\in N_{2\delta}(p_1)$, and let $u$ be a vertex of $p_1$ such that $d(v,u)\leq 2\delta$ and let $y$ be a geodesic from $u$ to $v$. Let $p_1 = p_1'p_1''$ where $p_1'$ ends at $u$. Since $g = Lab(\overline{p_1}) = Lab(\overline{p_1''}ys)Lab(\overline{s}\overline{q})$ and $g$ is a minimal length representative of $gH_1$, we have  $|g| = |p_1'|+|p_1''| \leq |p_1''|+|y| + |s|$, and so $|p_1'|\leq |y|+|s|\leq 2\delta+K$. Thus $|q|\leq |y|_|p_1'|\leq 4\delta+K$, and so $|p_h| \leq 2|q|+2\leq 8\delta + 2K + 2$. We have just shown that for every $h\in H_1\cap H_2^g$ the length $|h| \leq 2|q|+2\leq 8\delta$, which implies that $H_1\cap H_2^g$ is a finite group.
\end{proof}

\section{Finite stature}\label{sec: finite stature}
\begin{defn}[{\cite[Defn 1.1]{HuangWiseStature19}}]
 Let $G$ be a group and let $\Omega = \{H_{\lambda}\}_{\lambda\in \Lambda}$ be a collection of subgroups
of $G$. Then $(G, \Omega)$ has \emph{finite stature} if for each $H \in\Omega$, there are finitely many $H$-conjugacy classes of infinite subgroups of form $H \cap C$, where $C$ is an intersection
of (possibly infinitely many) $G$-conjugates of elements of $\Omega$.
\end{defn}
The main result of \cite{HuangWiseStature19} is the following.
\begin{thm}[{\cite[Thm 1.3]{HuangWiseStature19}}]
Let $G$ be the fundamental group of a graph of groups with finite underlying graph $\Gamma$. Suppose that
\begin{enumerate}
\item each  $G_v$ for $v\in V(\Gamma)$ is a hyperbolic, virtually compact special group,
\item each $G_e$ for $e\in E(\Gamma)$ is quasiconvex in its vertex groups,
\item $(G, \{G_v\}_{v\in V(\Gamma)})$ has finite stature.
\end{enumerate}
Then each quasiconvex subgroup of a vertex group of $G$ is separable in $G$. In particular, $G$ is residually finite.
\end{thm}

In particular, the first two conditions are automatically satisfied for any finite graph of finite rank free groups, and free groups are locally quasi-convex.
\begin{cor}\label{cor: graphs of free groups}
Let $G$ be the fundamental group of a graph of finite rank free groups. If $(G, \{G_v\}_{v\in V(\Gamma)})$ has finite stature, then every finitely generated subgroup of a vertex group of $G$ is separable. In particular, $G$ is residually finite.
\end{cor}

We also note the following characterization of finite stature in terms of edge stabilizers in the action of $G$ on the Bass-Serre tree associated to the splitting. All the stabilizers considered in this paper are \emph{pointwise} stabilizers.

\begin{lem}[{\cite[Lem 3.9, Lem 3.19]{HuangWiseStature19}}]\label{lem:HuangWise stabilizer lemma}
Let $T$ be the Bass-Serre tree of the splitting of $G$ as a graph of groups with the underlying graph $\Gamma$. Then $(G, \{G_v\}_{v\in V(\Gamma)})$ has finite stature if and only if for each $v\in V(\Gamma)$, there are only finitely many $G_v$-conjugacy classes of groups of the form $G_v\cap\bigcap_{e\in E} \Stab(e)$ where $E\subseteq E(T)$. 

Moreover, if all the vertex groups are hyperbolic and edge groups are quasiconvex, then it suffices to only consider finite subsets $E\subseteq E(T)$.
\end{lem}

We note that in the above statement, we can identify $G_v$ with $\Stab(\tilde v)$ for some fixed lift  $\tilde v\in V(T)$ of $v$. In fact, every conjugate of a vertex group of $G$ can be identified with $\Stab(\tilde v)$ for some $\tilde v\in V(T)$. We explain in more detail, how one can think of the intersections of conjugates of vertex groups. 

We will denote the pointwise stabilizer of a path $\rho$ in $T$ by $\Stab(\rho)$, i.e.\ $\Stab(\rho) = \bigcap_{e\in \rho}\Stab(e)$. Using the identification of $G_v$ with $\Stab(\tilde v)$, we can view $\Stab(\rho)$ as a subgroup of $G_v$ if $\tilde v$ is contained in $\rho$. To emphasize that, we will denote such a subgroup by $G_v\cap \Stab(\rho)$.
If $\rho, \rho'$ both pass through $\tilde v$ and $\rho\subseteq \rho'$, then $G_v\cap\Stab(\rho')\subseteq G_v\cap\Stab(\rho)$.

\begin{prop}\label{prop: conjugacy classes of path stabilizers} 
Let $G$ be a graph of $\delta$-hyperbolic groups with quasiconvex edge groups,
and let $T$ be its Bass-Serve tree. Then $(G, \{G_v\}_{v\in V(\Gamma)})$ has finite stature if and only if there are only finitely many $G_v$-conjugacy classes of groups of the form $G_{v}\cap \Stab(\rho)$ where $\rho$ is a finite path in $T$ passing through $\tilde v$.
\end{prop}

\begin{proof} Lemma~\ref{lem:HuangWise stabilizer lemma} implies that it suffices to show that there are finitely many conjugacy classes of groups of the form $G_{v}\cap\bigcap_{e\in E} \Stab(e)$ for $v\in V(\Gamma)$ and finite $E\subseteq E(T)$, if and only if there are only finitely many conjugacy classes of groups of the form $G_v\cap \Stab(\rho)$ where $\rho$ is a finite path that passes through $v$.

The forward implication is immediate. Let us assume there are only finitely many conjugacy classes of groups of the form $G_v\cap \Stab(\rho)$ where $\rho$ is a finite path passing through $\tilde v$.
The group $G_v\cap\bigcap_{e\in E} \Stab(e)$ is exactly the subgroup of $G$ stabilizing all the edges in $E$ (and in particular stabilizing $\tilde v$ which is an endpoint of some $e\in E$). In particular, $G_v\cap\bigcap_{e\in E} \Stab(e)$ can be realized as the subgroup of $G_v$ stabilizing the union of paths $\{\rho_e\}_{e\in E}$ where $\rho_e$ is the minimal path containing $v$ and the edge $e$, i.e.\ $G_v\cap\bigcap_{e\in E} \Stab(e) = G_v \cap \bigcap_{e\in E}\Stab (\rho_e) =  \bigcap_{e\in E}\left(G_v \cap \Stab (\rho_e)\right)$. Since there are only finitely many conjugacy classes of subgroups of the form $G_v \cap \Stab (\rho)$ and all such subgroups are quasiconvex in $G_v$ as finite intersections of quasiconvex subgroup, Lemma~\ref{lem: finitely many conjugacy classes of intersections} implies that there are also only finitely many conjugacy classes of their intersections.\end{proof}

We finish this section with the following observation that will allow us to work with certain finite index subgroups of the considered groups.

\begin{prop}[Passing to finite index supergroups]\label{prop:passing to finite index}
Let $G$ split as a graph of groups. If $G'$ is a finite index subgroup of $G$ such that $G'$ has finite stature with respect to the vertex groups in the induced graph of groups decomposition, then $G$ has finite stature with respect to its vertex groups.
 \end{prop}
\begin{proof}
This follows immediately from the characterization of finite stature in terms of the of number of orbits of based big trees in the sense of \cite[Def 3.7]{HuangWiseStature19}, see \cite[Lem 3.9]{HuangWiseStature19}.
\end{proof}

\section{Graphs of free groups}\label{sec: graphs of groups}
\subsection{Amalgamated products $A*_CB$ where $[B:C]=2$}
Let $G = A*_CB$ be an amalgamated product of finite rank free groups, where $[B:C] = 2$. Let $b\in B-C$, i.e.\ $bC$ is the nontrivial coset of $C/B$.

Let $T$ be the Bass-Serre tree of $G$ (metrized so that each edge of $T$ has length $1$).
The vertices of $T$ are of two kinds: infinite valence \emph{$A$-vertices}, corresponding to conjugates of $A$, and valence two  \emph{$B$-vertices} corresponding to conjugates of $B$. 
The edges of $T$ correspond to conjugates of $C$. We use the convention where $C^g$ denotes the conjugate $g^{-1}Cg$, so that $(C^g)^h = C^{gh}$. 

We start with the following observation.

\begin{lem}\label{lem:stabilizers of adjacent edges}
 An element $g\in G$ stabilizes an edge $e$ of $T$ if and only if $g$ stabilizes an adjacent edge $e'$ meeting $e$ at a $B$-vertex. 
\end{lem}
\begin{proof} 
Since the vertex incident to both $e$ and $e'$ has valence $2$, any element stabilizing one of the edges must stabilize the other one as well.
\end{proof}

\begin{remark}\label{rem: path stabilizers}
As a consequence of the lemma, we get that for every path $\rho'$ in $T$, $\Stab(\rho') = \Stab(\rho)$ where $\rho$ is the minimal path containing $\rho'$ that starts and ends at $A$-vertices.
\end{remark}
Thus we will only consider paths in $T$ starting and ending at $A$-vertices.
We continue measuring the length of paths with respect to the original metric on the tree, i.e.\ any two $A$-vertices are even distance away.

In the following lemma, we describe all the stabilizers of paths in $T$ joining two $A$-vertices. In our application, we will only need the statement for the paths of length at most $8$, so we give explicit description in those cases, but for completeness we also give the general statement for paths of arbitrary length. 

Let $\rho:[0,2\ell]\to T$ be a path joining two $A$-vertices. Since the length of $\rho$ is even, the middle point of $\rho$ is always a vertex in $T$. Depending on the parity of $\ell$, the middle vertex can be an $A$-vertex or a $B$-vertex. If $\ell$ is even, the middle vertex of $\rho$ is an $A$-vertex, and in the lemma below we will consider such paths where the middle vertex of $\rho$ is stabilized by $A$, and that the following vertex is stabilized by $B$ (see  initial subpath of length $4$ of the path in Figure~\ref{fig: length 6 path} for an example with $\ell=2$). If $\ell$ is odd, the middle vertex of $\rho$ is a $B$-vertex, and by conjugating $\Stab(\rho)$, and we will consider paths where the middle vertex is stabilized by $B$, and the vertex before is stabilized by $A$ (see Figure~\ref{fig: length 6 path} for an example where $\ell=3$). Note that those cases can be simultaneously described as satisfying $\Stab(\rho(2k)) = A$ and $\Stab(\rho(2k+1)) = B$ where $\ell=2k$ or $\ell = 2k+1$, depending on the parity of $\ell$.

\begin{lem}\label{lem:stabilizers as intersections}
Let $\rho:[0,2\ell]\to T$ be a length $2\ell$ combinatorial path in $T$ starting and ending at $A$-vertices.
Suppose that $\Stab(\rho(2k)) = A$ and $\Stab(\rho(2k+1)) = B$, where $\ell = 2k$ or $\ell=2k+1$ depending on the parity of $\ell$.

Then $\Stab(\rho)$ is of the form $K_\ell\subseteq C$ where: 
\begin{itemize} 
\item $K_1 = C$
\item $K_2 = C^{a_1} \cap C$ for some $a_1\in A$
\item $K_3 = C^{a_1}\cap C\cap C^{d_1b}$ for some $a_1, d_1\in A$
\item $K_4 = C^{a_2ba_1}\cap C^{a_1}\cap C \cap C^{d_1b}$ for some $a_1, a_2, d_1\in A$
\end{itemize}
and more generally, 
\begin{itemize} 
\item for $\ell = 2k+1$:\\
 $$K_{2k+1} =C^{a_kba_{k-1}b\dots ba_1}\cap C^{a_{k-1}b\dots ba_{1}}\cap \dots\cap C^{a_1}\cap C \cap C^{d_1b}\cap C^{d_2bd_1b}\cap \dots \cap C^{d_k\dots bd_{1}b}$$
for some $a_1, \dots, a_k, d_1, \dots, d_k\in A$;
\item for $\ell = 2k$:\\
$$K_{2k} =C^{a_kba_{k-1}b\dots ba_1}\cap C^{a_{k-1}b\dots ba_{1}}\cap \dots \cap C^{a_1}\cap C \cap C^{d_1b}\cap C^{d_2bd_1b}\cap \dots \cap C^{d_{k-1}\dots bd_{1b}}$$
 for some $a_1, \dots, a_k, d_1, \dots, d_{k-1}\in A$.
\end{itemize}
Additionally, we have the following, where $K_{\ell}$ and $K_{\ell}'$ denote two groups of the form as above (for possibly different choices of elements $a_i$'s and $d_i$'s).
\begin{itemize}
\item  if $\ell$ is odd, then $K_\ell = K_{\ell-1}\cap (K'_{\ell-1})^b$
\item  if $\ell$ is even, then $K_\ell = K_{\ell-1}\cap (K'_{\ell-1})^{a}$
\end{itemize}
\end{lem}

\begin{proof} 
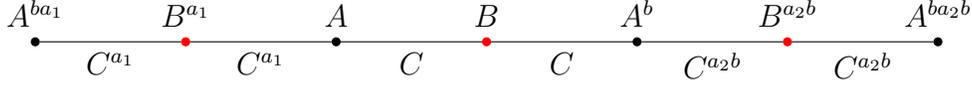
\begin{figure}
\begin{tikzpicture}
\node[circle, draw, fill, inner sep = 0pt,minimum width = 3pt, label=above:$A^{ba_1}$] (0) at (-4,0) {};
\node[circle, draw, fill, inner sep = 0pt,minimum width = 3pt, red, label=above:$B^{a_1}$] (1) at (-2,0) {};
\node[circle, draw, fill, inner sep = 0pt,minimum width = 3pt, label=above:$A$] (2) at (0,0) {};
\node[circle, draw, fill, inner sep = 0pt,minimum width = 3pt, red, label=above:$B$] (3) at (2,0) {};
\node[circle, draw, fill, inner sep = 0pt,minimum width = 3pt, label=above:$A^b$] (4) at (4,0) {};
\node[circle, draw, fill, inner sep = 0pt,minimum width = 3pt, red, label=above:$B^{d_1b}$] (5) at (6,0) {};
\node[circle, draw, fill, inner sep = 0pt,minimum width = 3pt, label=above:$A^{bd_1b}$] (6) at (8,0) {};
\draw (0) to node[below] {$C^{a_1}$} (1) to node[below] {$C^{a_1}$} (2) to node[below] {$C$} (3) to node[below] {$C$} (4) to node[below] {$C^{d_1b}$} (5) to node[below] {$C^{d_1b}$} (6);
\end{tikzpicture}
\caption{Every length $6$ path in the Bass-Serre tree of $A*_CB$ where \\ $[B:C] = 2$ is conjugate to the pictured path. The labels are the stabilizers. We note that two consecutive edges meeting at a $B$-vertex have the same stabilizers. See Lemma~\ref{lem:stabilizers as intersections}. Algebraically, this also follows from the fact that $C^b = C$, since $[B:C] = 2$. }\label{fig: length 6 path}\end{figure}

Since $\Stab(\rho) = \bigcap_{e\in \rho}\Stab(e)$, by analyzing the stabilizers of edges in $\rho$, we get the description of $\Stab(\rho)$ as required.

Let us now prove the second part of the statement. First assume that $\ell=2k+1$. Then 
\begin{align*}
K_{2k+1} = &(C^{a_kba_{k-1}b\dots ba_1}\cap C^{a_{k-1}b\dots ba_{1}}\cap \dots\cap C^{a_1}\cap C \cap C^{d_1b}\cap C^{d_2bd_1b}\cap \dots \cap C^{d_{k-1}\dots bd_{1}b})\cap \\
&\cap (C^{d_kbd_{k-1}b \dots bd_1}\cap \dots\cap C^{d_1}\cap C \cap C^{(a_1b^{-2})b}\cap C^{a_2b(a_1b^{-2})b}\cap \dots \cap C^{a_{k-1}b\dots (a_1b^{-2})b})^b
\end{align*}
We note that $a_1b^{-2}\in A$ since $b^2\in C$, so the expression above is indeed of the form $K_{2k}\cap (K'_{2k})^b$.
Similarly, when $\ell = 2k$, we get
\begin{align*}
K_{2k} = &(C^{a_{k-1}b\dots ba_{1}}\cap C^{a_{k-2}b\dots ba_{1}}\cap \dots\cap C^{a_1}\cap C \cap C^{d_1b}\cap C^{d_2bd_1b}\cap \dots \cap C^{d_{k-1}\dots bd_{1}b}) \cap \\
&\cap (C^{d_{k-2}\dots bd_{1}b a_1^{-1}}\cap \dots\cap C^{a_1^{-1}}\cap C\cap C^{a_2b}\cap \dots \cap C^{a_kb\dots ba_{2}b} )^{a_1}
\end{align*}
which gives as $K_{\ell} = K_{\ell-1}\cap (K'_{\ell-1})^{a}$ for $a=a_1$ as required.
\end{proof}

We emphasize that subgroups $K_{\ell}$ in the above statement are not uniquely defined, i.e.\ they depend on the choice of elements $a_i$ and $d_i$.

\subsection{Monochrome cycles preserving splittings}\label{sec:color respecting splittings}

We start with recalling the definition of a graph of spaces in the special cases where all the vertex and edge spaces are graphs.
A \emph{graph of graphs} $X(\Gamma)$ consist of the following data:
\begin{itemize}
\item a combinatorial graph $\Gamma$,
\item for every $v\in V(\Gamma)$, a metric graph $X_v$,
\item for every edge $e\in E(\Gamma)$, a metric graph $X_e$ such that $X_e \stackrel{\beta}{\simeq} X_{\bar e}$,
and an injective monotone map $\phi_e: X_e\to X_{\tau(e)}$.
\end{itemize}

Moreover, we say $X(\Gamma)$ is \emph{orientation preserving}, if all graphs $X_u$ for $u\in V(\Gamma)\cup E(\Gamma)$ are oriented, and maps $\phi_e$ are orientation preserving monotone maps (i.e.\ sending oriented paths to possibly trivial oriented paths).

We emphasize that we do not require $\phi_e$ to be a combinatorial map, but by Proposition~\ref{prop:factorization} we know that $\phi_e$ factors as the composition $X_e\to \overline X_{e} \to X_{\tau(e)}$ where the second map is a combinatorial immersion.

Let $[n]$ denote the set $\{1,\dots, n\}$. 
An \emph{edge coloring} of a metric graph $X$ is a maps $c:\{\text{1-cells of } X\}\to [n]$. 
We refer to $i\in [n]$ as \emph{colors}.
A cycle in a graph $X$ is \emph{monochrome} if each edge in the cycle has the same color. 
Suppose graphs $X, X'$ admit edge colorings $c,c'$ with colors $[n]$ respectively. A monotone map $\phi:X\to X'$ is \emph{color-preserving}, if $c'(\phi(e)) = c(e)$ for every $1$-cell $e$ of $X$. A \emph{color-preserving isomorphism} is a combinatorial map which is bijective on both vertex-sets and edge-sets, which is color-preserving.

\begin{defn}[Monochrome cycles preserving graph of graphs]
Fix $n\geq 1$ and for each $i\in [n] =\{1,\dots, n\}$ let $\ell_i\geq 1$.
Let $X(\Gamma)$ be a graph of graphs, where for each $u\in V(\Gamma)\cup E(\Gamma)$ there exists a coloring $c_u: \{\text{1-cells of }X_u\}\to [n]$, and if $u\in E(\Gamma)$ then $c_u = c_{\bar u}$.
A graph of graphs $X(\Gamma)$ is \emph{monochrome cycles preserving} if 
\begin{itemize}
\item for every $e\in E(\Gamma)$, $\phi_e$ is color-preserving, and 
\item for each $i\in [n]$ and each $u\in V(\Gamma)\sqcup E(\Gamma)$, the preimage $c_u^{-1}(i)\subseteq X_u$ is a disjoint union of embedded cycles,
\item for $e\in E(\Gamma)$, the map $\phi_e$ restricted to each cycle of color $i$ factors through a cycle of length $\ell_i$ in the factorization provided by Proposition~\ref{prop:factorization}.

\end{itemize}
\end{defn}

We can visualize such graphs of groups as having edges in vertex and edge graphs colored in a way that the induced colorings of edges in the edge graphs is consistent with respect to both adjacent vertex graphs. 
Note that in particular, each vertex and edge graph in a monochrome cycles preserving graph of graphs is a union of monochrome cycles.
The third condition can be thought of stating that each cycle of a given color in an edge graph has length $\ell_i$ in the metric induced by each vertex group. We note that this length does not need to correspond to the combinatorial length of that cycle, as the attaching maps $\phi_e$ do not need to be combinatorial. We make this (and more general) statement more precise in Lemma~\ref{lem:preserved length of colored cycles}. 
Instead of providing any examples now, we refer the reader to Section~\ref{sec: Artin} and splittings of Artin groups, induced by monochrome cycles preserving graph of graphs. They are the motivation for the above definition.

We will denote the associated graph of group by $G(\Gamma)$. 

\begin{lem}\label{lem:monochrome cycles}
Let $X(\Gamma)$ be a monochrome cycles preserving graph of groups.
\begin{enumerate}
\item For $j=1,2$, let $\overline Y_j\to X_v$ be a combinatorial immersion where for each color $i$ the subgraph of $\overline Y_i$ consisting of edges of color $i$ is a disjoint union of cycles of length $\ell_i$. Then for each color $i$ the subgraph of $\overline Y_1\otimes_{X_v}\overline Y_2$  consisting of edges of color $i$ is a disjoint union of cycles of length $\ell_i$.

\item Let $\psi: Y\to X_e$ be a combinatorial immersion. Let $\phi_e\cdot\psi: Y\to X_{\tau(e)}$ factor through $\overline Y$ and let $\phi_e\cdot\beta\cdot \psi: Y\to X_{\tau(\overline e)}$ factors through $\overline Y'$ in the factorization provided by Proposition~\ref{prop:factorization}. 
Then for each color $i$ the subgraph of $\overline Y$ consisting of edges of color $i$ is a disjoint union of cycles of length $\ell_i$, if and only if, for each color $i$ the subgraph of $\overline Y'$ consisting of edges of color $i$ is a disjoint union of cycles of length $\ell_i$.
\end{enumerate}
\end{lem}
\begin{proof}\item

\begin{enumerate}
\item We need to show that each $e$ of $Y_1\otimes_{X_v}Y_2$ of color $i$ is contained in a unique monochrome cycle of length $\ell_i$. For $j=1,2$, let $\pi_j:Y_1\otimes_{X_v}Y_2\to Y_j$ be the natural projection. Note that for $j=1,2$, $\pi_j(e)$ has color $i$ and by assumption it is contained in a unique monochrome cycle $C_j$ of length $\ell_i$. Thus $C_1, C_2$ lift to a monochrome cycle of length $\ell_i$ containing $e$.

\item 
For each color $i$ the subgraph of $\overline Y$ consisting of edges of color $i$ is necessarily a disjoint union of cycles and path segments, since $\overline Y$ embeds in some covering of $X_{\tau(e)}$. The same is also true for $Y$ for the same reason.

The map $X_e\to \overline X_e$ does not identify two vertices of $X_e$, which are both adjacent to edges of the same color, and consequently this property also holds for the map $Y\to \overline Y$. In particular, the subgraph of  $\overline Y$ consisting of edges of color $i$ has any path segments if and only if the subgraph of  $Y$ consisting of edges of color $i$ does. Thus the subgraph of  $\overline Y$ consisting of edges of color $i$ is a disjoint union of cycles if and only if the subgraph of  $\overline Y'$ consisting of edges of color $i$ is a disjoint union of cycles.

Finally, let $\overline C$ be a monochrome cycle of $\overline Y$ of color $i$, and let $C$ be the preimage of $\overline C$ in $Y$, which is also a monochrome cycle. Then $\overline C$ has length $\ell_i$ if and only if the map $\overline Y\to X_{\tau(e)}$ restricted to $\overline C$ is 1-1. This happens if and only if $Y\to X_e$ restricted to $C$ is 1-1. Thus each monochrome cycle of color $i$ in $\overline Y$ has length $\ell_i$ if and only if each monochrome cycle of color $i$ in $\overline Y'$ does.
\end{enumerate}
\end{proof}

In the next couple of Lemmas, we assume that $\rho\subseteq T$ is a path in the Bass-Serre tree of $G(\Gamma)$ passing through the vertex $\tilde v$, and an edge $\tilde e$ containing $\tilde v$. 
We identify the stabilizer $\Stab{\tilde v}$ with $G_v$ for some $v\in V(\Gamma)$, and the $\Stab(\tilde e)$ with $G_e$ for some $e \in E(\Gamma)$. 
We view the stabilizer $\Stab(\rho) =\bigcap_{e\subseteq \rho}\Stab(e)$ as a subgroup of $\Stab(\tilde e) = G_e$.

Since we are assuming that $G_v$ is the fundamental group of the graph $X_v$, the inclusion of $\Stab(\rho)$ in $G_v$ can be represented by the monotone map $\phi:Y_\rho \to X_v$, where $Y_\rho$ is the core of $\Stab(\rho)$ with respect to $X_e$, and the map is obtained by post-composition with $X_e\to X_v$. 
Let $\phi:Y_\rho \xrightarrow{\sigma} \overline Y_\rho\xrightarrow{\iota} X_v$ be a factorization of $\phi$ provided by Proposition~\ref{prop:factorization}. Graphs $Y_{\rho}$ and $\overline Y_{\rho}$ have natural coloring induced by their combinatorial immersions to $X_e$ and $\overline X_e$ respectively.

\begin{lem}\label{lem:preserved length of colored cycles}
Let $X(\Gamma)$ be a monochrome cycle preserving graph of groups. Then for every finite path $\rho$ in the Bass-Serre tree of the associated group $G(\Gamma)$, for each color $i$ the subgraph of $\overline Y_\rho$ consisting of edges of color $i$ is a disjoint union of cycles of length $\ell_i$.
\end{lem}

\begin{proof} 
Since $\Stab(\rho) = \bigcap_{e\in\rho} \Stab(e)$, we can obtain $\overline Y_\rho$ by a finite sequence and of fiber product of graphs $\overline Y_1\otimes_{X_v}\overline Y_2$ and moving between the factorizations of intermediate graphs $Y$ combinatorially immersing in some $X_e$ with respect to two maps to the vertex spaces $X_{\tau(e)}, X_{\tau(\overline e)}$ 
By Lemma~\ref{lem:monochrome cycles} those operations preserve the property that each subgraph of color $i$ is a disjoint union of cycles of length $\ell_i$. Thus the resulting graph $\overline Y_\rho$ has this property.
\end{proof}

As a consequence of Lemma~\ref{lem:preserved length of colored cycles}, we can view every $\overline Y_\rho$ as the $1$-skeleton of a $2$-complex $\widetriangle Y_\rho$ obtained by attaching $\ell_i$-gons of color $i$ along each monochrome cycle of color $i$.

If $\rho\subseteq \rho'$, then $\Stab(\rho')\subseteq \Stab(\rho)$ and so there is a combinatorial immersion $\overline Y_{\rho'}\to \overline Y_{\rho}$ over $X_v$.

\begin{lem}\label{lem:simply connected Y triangle}
 Suppose $\rho\subseteq \rho'$ and $\widetriangle Y_\rho$ is simply connected. Then the combinatorial immersion $\overline Y_{\rho'}\to \overline Y_{\rho}$ is  an embedding of a subgraph.
\end{lem}
\begin{proof}
Since $\widetriangle Y_{\rho}$ is simply connected, it follows that $\widetriangle Y_{\rho'}\subseteq \widetriangle Y_{\rho}$, and so $\overline Y_{\rho'}\subseteq \overline Y_{\rho}$.
\end{proof}

\section{Finite stature in triangle Artin groups}\label{sec: Artin}
\subsection{The statement}
A \emph{triangle Artin group} is given by the presentation
$$G_{MNP} = \langle a, b, c\mid (a, b)_{M} = (b, a)_{M}, (b,c)_{N} = (c,b)_{N}, (c,a)_{P} = (a,c)_{P}\rangle,$$
where $(a,b)_M$ denote the alternating word $aba\dots$ of length $M$.

The following theorem describes a splitting of $G_{MNP}$ as an amalgamated product of free groups, where the map from the amalgamating subgroup to the vertex groups is described in terms of maps between graphs.

\begin{thm}[{\cite[Cor 4.13]{JankiewiczArtinRf}}]\label{thm:splitting}
Let $G_{MNP}$ be an Artin group where $M,N,P\geq 3$.  Then $G_{MNP} = A*_CB$ where $A\simeq F_3$, $B\simeq F_4$ and $C\simeq F_7$, and $[B:C] = 2$. The map $C\to A$ is induced by the map $\phi: X_C \to X_A$ pictured in Figure~\ref{fig:mapCtoA}, and the map $C\to B$ is induced by the quotient of the graph $X_C$ by a $\pi$ rotation.
\begin{figure}
\includegraphics[scale=0.4]{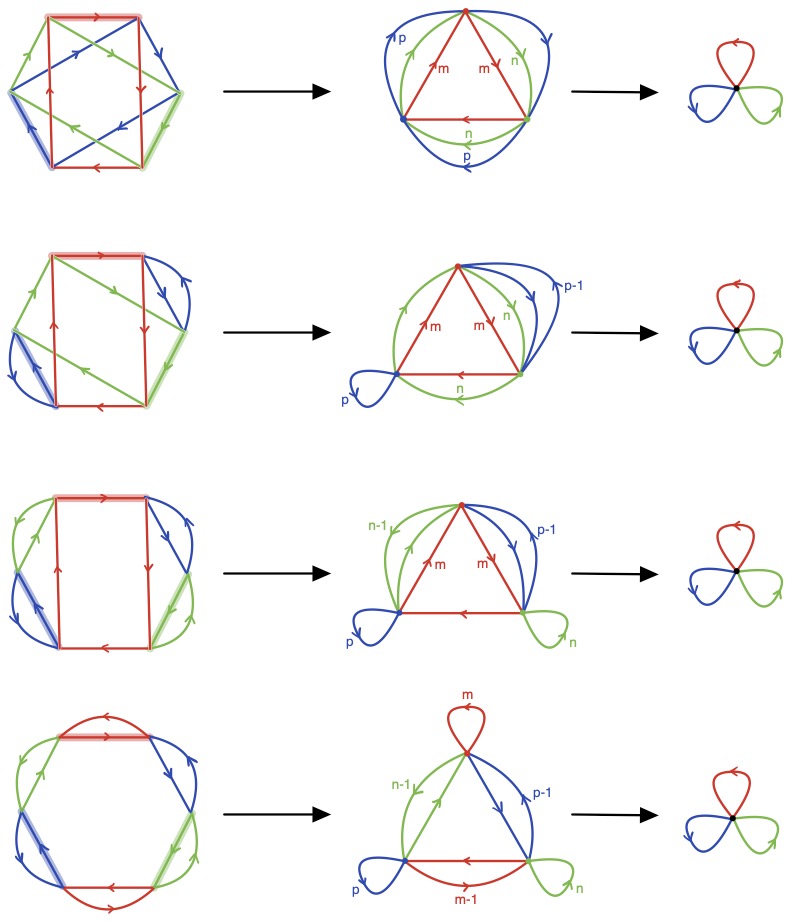}
\caption{The map $\phi:X_C\xrightarrow{\sigma}\overline{X}_C\xrightarrow{\iota} X_A$ when (1) none, (2) one, (3) two or (4) all of $M, N, P$ are even, respectively.
Specifically, $M=2m$ or $2m+1$, $N=2n$ or $2n+1$, and $P=2p$ or $2p+1$. 
We use the convention where the edge labelled by a number $k$ is a concatenation of $k$ edges of the given color. The thickened edges in $X_{C}$ are the ones that get collapsed to a vertex in $\overline X_C$}\label{fig:mapCtoA}
\end{figure}
\end{thm}

\begin{thm}[{\cite[Prop 2.8]{JankiewiczArtinSplittings}}]\label{thm:splitting2}
Let $G_{MNP}$ be an Artin group where $M,N\geq 4$  and $P=2$. 
\begin{itemize}
\item If at least one of $M,N$ is odd, then $G_{MNP} = A*_CB$ where $A\simeq F_2$, $B\simeq F_3$ and $C\simeq F_5$, and $[B:C] = 2$. The map $C\to A$ is induced by the map $\phi: X_C \to X_A$ pictured in Figure~\ref{fig:2MN}, and the map $C\to B$ is induced by the quotient of the graph $X_C$ by a $\pi$ rotation.
\begin{figure}
\includegraphics[scale=0.35]{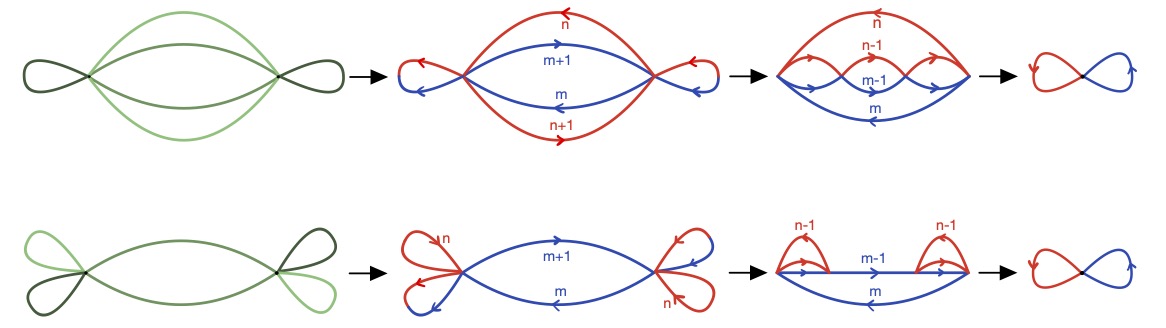}
\caption{The map $\phi:X_C\xrightarrow{id}X_C \xrightarrow{\sigma}\overline{X}_C\xrightarrow{\iota} X_A$ 
when $P=2$, $M=2m+1\geq 5$, and (top) $N=2n+1\geq 5$, (bottom) $N=2n\geq 4$, respectively. The use of colors in the leftmost graphs represents the $\pi$-rotation of $X_C$.}\label{fig:2MN}
\end{figure}
\item If both $M,N$ are even then $G_{MNP} = A*_B$ where $A\simeq F_2$, $B\simeq F_3$. The two maps $B\to A$ are induced by the maps $\phi_1, \phi_2: X_B \to X_A$ pictured in Figure~\ref{fig:2eveneven}.
\begin{figure}
\includegraphics[scale=0.35]{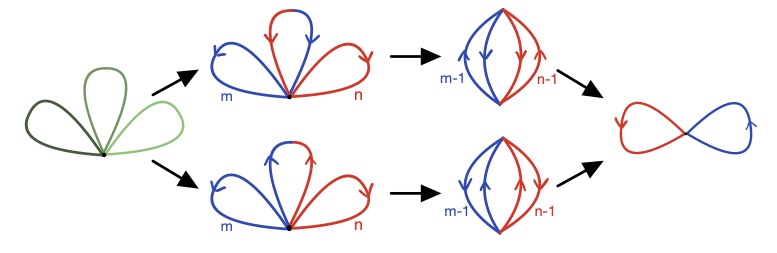}
\caption{The maps $\phi_i: X_B\xrightarrow{id}X_B \xrightarrow{\sigma}\overline{X}_B\xrightarrow{\iota} X_A$ for $i=1,2$, when $M=2m\geq 4$, $N=2n\geq 4$, and $P=2$.}\label{fig:2eveneven}
\end{figure}
\end{itemize}
\end{thm}

Here is a precise statement of the main theorem of this paper (Theorem~\ref{thm: main}).
\begin{thm}
Let $G_{MNP}$ be a triangle Artin group where $M\leq N\leq P$ and either $M>2$, or $N>3$. Then $G_{MNP}$ has finite stature with respect to $\{A\}$, where $A$ is as described in Theorem~\ref{thm:splitting} or Theorem~\ref{thm:splitting2} respectively. All finitely generated subgroups of $A$ are separable in $G_{MNP}$, and in particular $G_{MNP}$ is residually finite.
\end{thm}

In subsections~\ref{sec:at least one even}, \ref{sec:two 4}, \ref{sec:all odd} and \ref{sec:2} we will prove groups $G_{MNP}$ as above have finite stature with respect to $\{A\}$ by analyzing various cases (see Proposition~\ref{prop: at least one even and not 2m+1,4,4}, Proposition~\ref{cor:442m+1}, Proposition~\ref{cor: all odd}, and Proposition~\ref{cor: 2MN}). The separability of finitely generated subgroups of $A$ will then follow from Corollary~\ref{cor: graphs of free groups}.

\subsection{Some facts about the splittings of Artin groups}
We start with some facts that will be used in the next sections.
We first focus on the cases where $G_{MNP}$ splits as $A*_CB$.
Let $\beta: X_{C}\to X_C$ be the $\pi$-rotation, as in Theorem~\ref{thm:splitting} or Theorem~\ref{thm:splitting2} respectively. 
A choice of a path between $x\in X_C$ and $\beta(x)\in X_C$ determines an element $b\in B-C$, such that the induced homomorphism $C \to C$ is the conjugation by $b$. 
We emphasize that $\beta^2 $ is the identity map.
%sentence added
Figure~\ref{fig:mapCtoA} and Figure~\ref{fig:2MN} illustrate the factorization $\phi = \iota \circ \sigma$ from Proposition~\ref{prop:factorization}.
We denote $\sigma(X_C) = \overline X_C$. 

We will also extend the definition of $\sigma$ to any combinatorial immersion of $X_C$ (and abuse the notation) in the following way. Given a combinatorial immersion $Y\to X_C$, let $\overline Y\to \overline X_C$ be a combinatorial immersion, and let $\sigma:Y\to \overline Y$ be a composition of edge-subdivisions, Stallings' folds, and edge-collapses, which locally coincides with $\sigma: X_C\to \overline X_C$. In particular, the following diagram commutes.
\[\begin{tikzcd}
Y \arrow{r}{\sigma} \arrow{d}{} & \overline Y \arrow{d}{} \\
X_C\arrow{r}{\sigma}& \overline X_C.
\end{tikzcd}
\]

We note the following. 
\begin{lem}\label{lem:1-to-1 correspondence between cores}
The map $\sigma: Y\to \overline Y$ is a homotopy equivalence for every combinatorial immersion $Y\to X_C$.

For each subgroup $H\subseteq C$ there is a one-to-one correspondence between the core $Y\to X_C$ of $H$ with respect to $X_C$ and the core $\overline Y\to \overline X_C$ of $H$ with respect to $\overline X_C$, where $\overline Y = \sigma(Y)$ as above.
\end{lem}

\begin{proof} The map $\sigma: X_C\to \overline X_C$ is obtained as a sequence of edge-subdivisions and edge-collapses of the edges. By analyzing each of the cases in Figure~\ref{fig:mapCtoA} and Figure~\ref{fig:2MN}, we note that we never collapse a loop. Thus, by discussion in Section~\ref{sec: maps}, $\sigma: X_C\to \overline X_C$ is a homotopy equivalence. Similarly, any induced map $Y\to \overline Y$ is also obtained as a sequence of edge-subdivisions and edge-collapses of the edges that are not loops, and hence $\sigma: Y\to \overline Y$ is a homotopy equivalence. By construction $Y\to X_C$ is the core of some subgroup $H\subseteq \pi_(X_C)$ with respect to $X_C$ if and only if $\overline Y\to \overline X_C$ is the core of $H$ with respect to $\overline X_C$.
\end{proof}

We will use the notation $\sigma^{-1}(\overline Y)$ to denote $Y$ such $\overline Y = \sigma(Y)$. 

\begin{lem}\label{lem: conjugate core}
Let $H\subseteq C$ be a subgroup, and let $Y\to X_C$ be its core with respect to $X_C$. Then $Y\to X_C\xrightarrow{\beta} X_C$ is the core of $H^b\subseteq C$. Moreover, 
\end{lem}
\begin{proof}
Indeed, $Y\to X_C$ induces the inclusion $H\to C$ and $X_C\xrightarrow{\beta} X_C$ induces the conjugation by $b$.
\end{proof}

When we consider the composition of the map $Y\to X_C\xrightarrow{\beta} X_C$ with $\sigma:X_C\to \overline X_C$ we will again abuse the notation and write $\sigma\beta(Y)$ to represent the map $Y\to \overline X_C$.
%added this

The following lemma will allow us to apply Proposition~\ref{prop: conjugacy classes of path stabilizers}.

\begin{lem}\label{lem:computing stabilizers}
Let $T$ be the Bass-Serre tree of the splitting $G_{MNP} = A*_CB$, and $\rho$ be a finite path in $T$ of length $2\ell$ between a pair of  $A$-vertices and such that $\Stab(\rho(2k)) = A$ where $\ell = 2k$ or $2k+1$.
The $A$-conjugacy class of the stabilizer $\Stab(\rho)$ is represented by a combinatorial immersion $\overline Y_{\ell}\to \overline X_C\to X_A$, where the corresponding $Y_\ell$ is defined recursively:
\begin{itemize}
\item $Y_1 = X_C$,
\item $Y_{\ell}$ is a connected component of $\sigma^{-1} (\overline{Y}_{\ell-1}\otimes_{X_A} \overline{Y}_{\ell-1})$ for even $\ell$, 
\item $Y_{\ell}$ is a connected component of $\sigma^{-1} (\overline Y_{\ell-1}\otimes_{X_A} \sigma\cdot \beta(Y_{\ell-1}))$ for odd $\ell$
\end{itemize}
The map $\overline Y_{\ell-1}\to X_A$ in the recursive definition above is obtained by composing the map $\overline Y_{\ell-1}\to \overline X_C$ with the map $\overline X_C\to X_A$.
\end{lem}
\begin{proof}
Let $\rho$ be a path of length $2\ell$. By Lemma~\ref{lem:stabilizers as intersections}, $\Stab(\rho)$ is equal to a group $K_{\ell}$ defined recursively as 
\begin{itemize}
\item $K_1 = C$,
\item  if $\ell$ is even, then $K_\ell = K_{\ell-1}\cap (K'_{\ell-1})^{a_1}$.
\item  if $\ell$ is odd, then $K_\ell = K_{\ell-1}\cap (K'_{\ell-1})^b$,
\end{itemize}
Clearly, $Y_1 = X_C\to X_C$ is the core of $K_1$ $= C$ with respect to $X_C$. 
For even $\ell$, $K_\ell = K_{\ell-1}\cap (K'_{\ell-1})^{a}$, so by Lemma~\ref{lem:fiber product} the core of $K_{\ell}$ with respect to $\overline X_C$ is $\overline Y_{\ell-1}\otimes_{X_A} \overline Y_{\ell-1}$, and by Lemma~\ref{lem:1-to-1 correspondence between cores} the core of $K_{\ell}$ with respect to $X_C$ is $\sigma^{-1}( \overline Y_{\ell-1}\otimes_{X_A} \overline Y_{\ell-1})$.
For odd $\ell$, $K_\ell = K_{\ell-1}\cap (K'_{\ell-1})^b$, so by Lemma~\ref{lem:fiber product} and Lemma~\ref{lem: conjugate core} the core $K_{\ell}$ with respect to $\overline X_C$ is $\overline Y_{\ell-1}\otimes_{X_A} \sigma\cdot \beta (Y_{\ell-1})$, and by Lemma~\ref{lem:1-to-1 correspondence between cores} the core of $K_{\ell}$ with respect to $X_C$ is $\sigma^{-1} (\overline Y_{\ell-1}\otimes_{X_A} \sigma\cdot \beta (Y_{\ell-1}))$.
\end{proof}

We emphasize that a group $K_{\ell}$ is not uniquely determined, and similarly a graph $Y_{\ell}$ is not uniquely determined, as at each step in this recursive construction there might be multiple connected components to choose from.
Since a sequence of group $C = K_1\supseteq K_2\supseteq K_3 \supseteq \dots$ form a descending chain, we have a corresponding sequence of combinatorial immersions $\dots \xrightarrow{} Y_3 \to Y_2\to Y_1  = X_C$. 

\begin{lem}\label{lem: arbitrary paths}
Let $\rho$ be an arbitrary finite path in $T$ of length $2\ell$ between the pair of $A$-vertices containing a vertex whose stabilizer is $A$. Then the $A$-conjugacy class of $\Stab(\rho)$ is represented by a $\overline Y_\ell$ or $\sigma\beta(Y_{\ell})$.
\end{lem}
\begin{proof}  Let $\ell = 2k$ or $2k+1$. The proof is by induction on the half of the distance $d$ from the vertex of stabilized by $A$ to the ``middle'' $A$-vertex $\rho(2k)$ of $\rho$ (since any two $A$ vertices are at even distance, $d$ measures the number of steps between ``consecutive'' $A$-vertices). When $d = 0$, i.e.\ $\rho(2k)$ is stabilized by $A$, then we are in the setting of Lemma~\ref{lem:computing stabilizers} and the $A$-conjugacy class of $\Stab(\rho)\subseteq A$ is represented by $\overline Y_{\ell}$. Suppose that for every length $2\ell$ path $\rho'$ between $A$-vertices where the distance from $\rho'(2k)$ to the vertex stabilized by $A$ equals $d-1$, the $A$-conjugacy class of $\Stab(\rho')$ is represented by $\overline Y_{\ell}$ or $\sigma\beta(Y_{\ell})$.

Now let $\rho$ be a length $2\ell$ path between $A$-vertices where the distance from $\rho(2k)$ to the vertex stabilized by $A$ equals $d$. Then there exist $b \in B-C$ and $a\in A$ (possibly $a=1$) such that $\rho = a^{-1}b^{-1}\rho'$, where $\rho'$ is another length $2\ell$ path with the distance from $\rho'(2k)$ to the vertex stabilized by $A$ equal $d-1$. By the inductive assumption the $A$-conjugacy class of $\Stab(\rho')$ is represented by $\overline Y_\ell$ or $\sigma\beta(Y_{\ell})$. The $A$-conjgacy class of $\Stab(\rho) = a^{-1}b^{-1}\Stab(\rho')ba$ is represented by $\sigma\beta (Y_\ell)$ or $\sigma\beta\sigma^{-1}\sigma\beta(Y_{\ell})  = \sigma \beta^2(Y_{\ell}) = \overline Y_{\ell}$. The last equality holds since $\beta^2$ is the identity.
\end{proof}

\subsection{Representing combinatorial maps as colored graphs}
By orienting and coloring all the edges of $X_A$ with distinct colors, we can represent the combinatorial immersion $\overline Y_{\rho} \xrightarrow{\iota} X_A$ as the graph $\overline Y_{\rho}$ whose edges are oriented and colored by the colors of the edges of $X_v$. From now on, we will mostly view $\overline Y_\rho$ as graphs together with orientation and coloring of edges, which means that such a graph encodes a map $\overline Y_\rho \to X_A$.

\begin{lem}\label{lem: finitely many colored graphs}
Suppose that there are only finitely many orientation and color preserving
isomorphism types of graphs $Y_{\ell}$ for any $\ell\geq 1$. Then $G_{MNP}$ has finite stature with respect to $\{A\}$.

In particular if there exists $k\geq 1$ such that every map $\overline Y_{k+2}\to \overline Y_k$ is an embedding of a subgraph, then $G_{MNP}$ has finite stature with respect to $\{A\}$.
\end{lem}

\begin{proof}
We first prove the first statement. 
The orientation and color preserving isomorphism types of graphs $Y_{\ell}$ correspond to combinatorial immersions $Y_{\ell}\to X_A$. 
By Lemma~\ref{lem:computing stabilizers} and Lemma~\ref{lem:stabilizers as intersections} there are finitely many $A$-conjugacy classes of $\Stab(\rho)$ for finite paths $\rho$ in $T$ joining two $A$-vertices with $\Stab(\rho(2k)) = A$ where $\ell = 2k$ or $2k+1$. By Lemma~\ref{lem: arbitrary paths} there are also only finitely many $A$-conjugacy classes of $\Stab(\rho)$ for an arbitrary path $\rho$ between $A$-vertices, and passing through the vertex stabilized by $A$.
By Remark~\ref{rem: path stabilizers}, every for every finite path $\rho'$ in $T$, $\Stab(\rho') = \Stab(\rho)$ where $\rho$ is the shortest path containing $\rho'$ joining two $A$-vertices.
We conclude that there are only finitely many $A$-conjugacy classes of groups of the form $A\cap \Stab(\rho)$ where $\rho$ is any finite path passing through the vertex . Thus the assumptions of Proposition~\ref{prop: conjugacy classes of path stabilizers} are satisfied. 
We deduce that $G_{MNP}$ has finite stature with respect to $\{A\}$.

Now let $k\geq 1$ such that $Y_{k+2}\to Y_k$ is an inclusion. Since $Y_{k+2}$ is obtained from $Y_k$ in two steps as described in Lemma~\ref{lem:computing stabilizers}, we deduce that $Y_{k+2(i+1)}\to Y_{k+2i}$ is an inclusion for each $i\geq 0$. In particular, there can only be finitely many orientation and color preserving isomorphism types of graphs $Y_{k+2i}$ since $Y_{k}$, as a finite graph, has only finitely many subgraphs. Using the formula for $Y_{k+2i+1}$ from Lemma~\ref{lem:computing stabilizers} we deduce that there are finitely many isomorphism types of graphs for any $\ell\geq 1$. The conclusion follows from the first part of the lemma.
\end{proof}

\subsection{Monochrome cycle preserving structure of splitting of Artin groups}

\begin{prop}\label{prop:artin monchrome cycles preserving}
Let $G_{MNP}$ be an Artin group $M,N,P\geq 3$. Then $G_{MNP}$ has a subgroup $G'$ of index at most $2$ that is the fundamental group of a monochrome cycles preserivng graph of graphs $X_A\xleftarrow{\phi} X_C\xrightarrow{\beta\cdot\phi } X_A$.
\end{prop}

\begin{proof} Let $G_{MNP} = A*_CB$ as in Theorem~\ref{thm:splitting}. Then $G_{MNP}$ has an index $2$ subgroup $G'$ which splits as $A*_CA$. The associated graph of graphs has two vertices with each vertex graph being a copy of $X_A$, and one edge graph $X_C$. We choose the coloring of $c_A:X_A\to \{\text{red}, \text{green}, \text{blue}\}$, where each loop has distinct color, as in Figure~\ref{fig:mapCtoA}. Those figure also show how the coloring is $c_C:X_C\to \{\text{red}, \text{green}, \text{blue}\}$ is defined. The two maps $X_C\to X_A$ differ by precomposing one with the automorphism $\beta$ of $X_C$. In particular, both maps $X_C\to X_A$ are orientation and color preserving, and the preimage of each color in $X_C$ is a union of disjoint embedded cycles. Moreover, the maps $X_C\to X_A$ both factor through $\overline X_C$, and in particular, both maps restricted to each cycle factors through a cycle of length $Q$ if $Q$ is odd, and $Q/2$ is $Q$ is even, for $Q = M,N,P$ respectively.
Thus the graphs of graphs $X_A\xleftarrow{\phi} X_C\xrightarrow{\beta\cdot\phi } X_A$ is orientation and monochrome cycles preserving.
\end{proof}

Every finite path $\rho$ in the Bass-Serre tree of $G_{MNP} = A*_CB$ joining a pair of $A$-vertices can be also thought of as a path in the Bass-Serre tree of the index $2$ subgroup $G' = A*_CA$ of $G_{MNP}$. By Proposition~\ref{prop:artin monchrome cycles preserving} above and Lemma~\ref{lem:preserved length of colored cycles}, for the combinatorial immersion  $Y_\rho \to X_A$ of $\Stab(\rho)$ the associated graph $\overline Y_\rho$ is a union of monochrome cycles, where each cycle of color $i$ has length $\ell_i$. We can denote the the $2$- complex obtained from $Y_\rho$ by attaching $2$-cells whose boundaries have color $i$ and length $\ell_i$ by $\widetriangle Y_\rho$, as in Section~\ref{sec:color respecting splittings}. 
\begin{notn}\label{notn}
We now switch to the use of notation of Lemma~\ref{lem:computing stabilizers}, where the graph $Y_\rho$ is denoted by $Y_{\ell}$ where $2\ell = |\rho|$, and the associated $K_{\ell}$ is the stabilizer $\Stab(\rho)$. We will also write $\widetriangle Y_{\ell}$ for $\widetriangle Y_\rho$. Once again, we remind that $Y_{\ell}$, $K_{\ell}$ depend not only on $\ell$, but also the choice of parameters $a_i, d_i$ in their definition, which are equivalent to the choice of $\rho$.
\end{notn}

\begin{lem}\label{lem: finite stature of twisted doubles}
If for some $\ell\geq 1$ a complex $\widetriangle Y_{\ell}$ is simply connected, then for every $\overline Y_{\ell+2}$, the combinatorial immersion  $\overline Y_{\ell+2}\to \overline Y_\ell$ is an embedding of a subgraph. In particular, if there exists $\ell\geq 1$ such that every $\widetriangle Y_\ell$ is simply connected, then $G_{MNP}$ has finite stature with respect to $\{A\}$.
\end{lem}
\begin{proof}
The first statement follows directly from Lemma~\ref{lem:simply connected Y triangle}.
Since there are only finitely many orientation and color preserving isomorphism types of $\overline Y_k$, there are also only finitely many orientation and color preserving isomorphism types of their subgraphs. Thus if all $Y_k$ are simply-connected, there are only finitely many orientation and color preserving isomorphism types of graphs that $\overline Y_{\rho}$ might have. It follows that there are only finitely many conjugacy classes of the groups of the form $G_{\tilde v}\cap \Stab(\rho)$. By Proposition~\ref{prop: conjugacy classes of path stabilizers} $G'$ has finite stature with respect to both copies of $A$. By Proposition~\ref{prop:passing to finite index} $G_{MNP}$ also has finite stature with respect to $\{A\}$.
\end{proof}

In the next subsections we apply Lemma~\ref{lem: finitely many colored graphs} or Lemma~\ref{lem: finite stature of twisted doubles} to prove that all the large type triangle Artin group have finite stature. We consider three cases:
\begin{enumerate}
\item[(Sec \ref{sec:at least one even})] at least one $M,N,P\geq 3$ is even and $\{M,N,P\}  \neq \{2m+1,4,4\}$ for $m\geq 1$,
\item[(Sec \ref{sec:two 4})]  $\{M,N,P\}  = \{2m+1,4,4\}$ where $m\geq 1$,
\item[(Sec \ref{sec:all odd})] all $M,N,P$ are odd and $\geq 3$.
\end{enumerate}
We also consider the case where one of the exponents is $2$, and the other two are both strictly greater than $3$:
\begin{enumerate}\setcounter{enumi}{3}
\item[(Sec \ref{sec:2})] $\{M,N,P\}$ where $M,N\geq 4$ and $P=2$. 
\end{enumerate}

The goal in all the cases is to prove that there are only finitely many orientation and color preserving isomorphism types of graphs $\overline Y_{\ell}$. In the remaining sections, we will just say an ``isomorphism'' in reference to an ``orientation and color preserving isomorphism''.

\subsection{Case where at least one of $M,N,P\geq 3$ is even and $\{M,N,P\} \neq \{2m+1,4,4\}$}\label{sec:at least one even}
In the next proof, we continue to use Notation~\ref{notn}.

\begin{prop}\label{prop: at least one even and not 2m+1,4,4}
Suppose 
$M, N, P\geq 3$ and at least one of them is even, but $\{M,N,P\} \neq \{2m+1,4,4\}$.
Then $G_{MNP}$ has finite stature with respect to $\{A\}$, where $A$ is as in Theorem~\ref{thm:splitting}.
\end{prop}
\begin{proof} By Theorem~\ref{thm:splitting} in all the cases listed in the statement, $G_{MNP}$ splits as an amalgamated product $A*_C B$ of finite rank free groups where $[B:C] = 2$, which by Proposition~\ref{prop:artin monchrome cycles preserving} is virtually the fundamental group of a monochrome cycles preserving graph of graphs. By \cite[Lem 5.2, 5.3, 5.4]{JankiewiczArtinRf} (see also \cite[Rem 5.5]{JankiewiczArtinRf}) $\widetriangle Y_2$ is simply-connected, where $Y_2\to X_C$ if the core of $C\cap C^g$ with respect to $X_C$, as in Lemma~\ref{lem:computing stabilizers}. By Lemma~\ref{lem: finite stature of twisted doubles} $G_{MNP}$ has finite stature with respect to $\{A\}$.
\end{proof}
We note that the residual finiteness of the Artin groups considered above was also proven in \cite{JankiewiczArtinRf}.

\subsection{Case where $\{M,N,P\}  = \{2m+1,4,4\}$}\label{sec:(2m+1,4,4)}\label{sec:two 4}
We continue to use Notation~\ref{notn}.

\begin{lem}\label{lem:44(2p+1)}
Let $\{M,N,P\}  = \{2m+1,4,4\}$. Every graph $\overline Y_2$ is either the left graph in Figure~\ref{fig:44(2p+1)}, or has simply connected $\widetriangle Y_2$.
Every graph $\overline Y_3$ is either the right graph in Figure~\ref{fig:44(2p+1)}, or has simply connected $\widetriangle Y_3$.
The map $\overline Y_4\to \overline Y_2$ is always an embedding of a subgraph.
\begin{figure}
\includegraphics[scale=0.4]{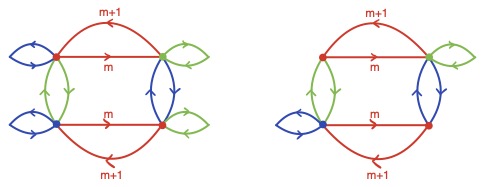}
\caption{$(M,N,P)  = (2m+1,4,4)$. The graph on the left is the fiber product $\overline Y_2 = \overline X_C\otimes_{X_A}\overline X_C$. The graph on the right is $\sigma\beta(Y_2)\otimes_{X_A} \overline Y_2$.}\label{fig:44(2p+1)}
\end{figure}
\end{lem}

\begin{proof}
By Theorem~\ref{thm:splitting} in all the cases listed in the statement, $G_{MNP}$ splits as an amalgamated product $A*_C B$ of finite rank free groups where $[B:C] = 2$, which by Proposition~\ref{prop:artin monchrome cycles preserving} is virtually the fundamental group of a monochrome cycles preserving graph of graphs.

By Lemma~\ref{lem:computing stabilizers}, $\overline Y_2$ is computed as a connected component of the fiber product $\overline Y_1\otimes_{X_A}\overline Y_1$, which has been done in \cite[Lem 5.3]{JankiewiczArtinRf}. If $\widetriangle Y_2$ is simply-connected, then $\overline Y_4\to \overline Y_2$ is an embedding of a subgraph for every $Y_4$, by Lemma~\ref{lem: finite stature of twisted doubles}.

In the case where $\widetriangle Y_2$ is not simply connected, $Y_2$ is the graph on the left in Figure~\ref{fig:44(2p+1)}.
This graph has an order $2$ isomorphism which can be represented by swapping the top-left vertex with the bottom-left vertex, and the top-right vertex with the bottom-right vertex, and extending appropriately to the edges. 
The two maps $\overline Y_2\to \overline X_C$ (corresponding to the projection onto two components of the fiber product) differ by precomposing one with this symmetry, and so both are represented by the first vertical arrow in Figure~\ref{fig:proof 44M}. 
Lemma~\ref{lem:1-to-1 correspondence between cores} ensures that $Y_2\to X_C$ can be computed, which is done in the second vertical arrow in Figure~\ref{fig:proof 44M}. 
Then the rest of Figure~\ref{fig:proof 44M} represent the computation of $\sigma\beta(Y_2)\to \overline X_C$.
Finally, by Lemma~\ref{lem:computing stabilizers}, $\overline Y_3$  is computed as the fiber product  $\sigma\beta(Y_2)\otimes_{X_A} \overline Y_2$, i.e.\ the fiber product of the left top and the right top graphs in Figure~\ref{fig:proof 44M}. 
We deduce that $\overline Y_3$ either has simply connected $\widetriangle Y_3$, or it is the right graph in Figure~\ref{fig:44(2p+1)}.
\begin{figure}
\includegraphics[scale=0.4]{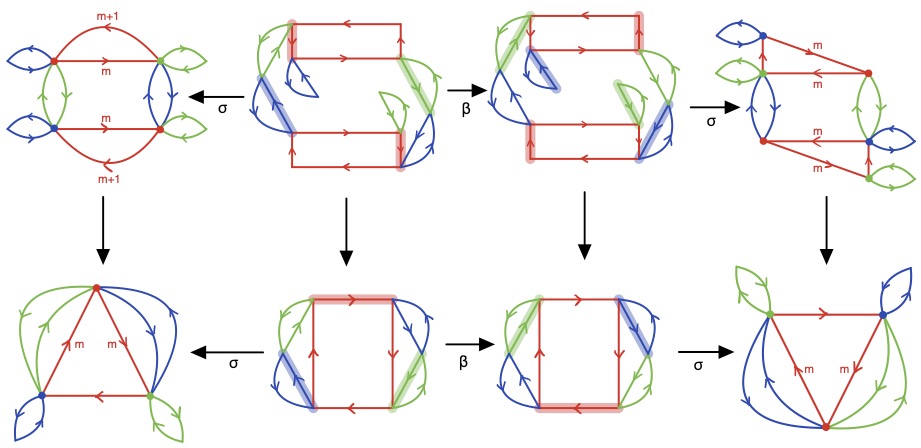}
\caption{$(M,N,P)  = (2m+1,4,4)$. The vertical arrows are respectively: $\overline Y_2 \to \overline X_C$, $Y_2\to X_C$, $\beta(Y_2)\to X_C$, and $\sigma\cdot \beta(Y_2) \to \overline X_C$. 
}
\label{fig:proof 44M}
\end{figure}

If $\widetriangle Y_3$ is simply-connected, then so is $\widetriangle Y_4$ and $\overline Y_4\to \overline Y_2$ is  an embedding of a subgraph, as required.
Otherwise, $\overline Y_4$ is a connected component of $\overline Y_3\otimes_{X_A} \overline Y_3$ by Lemma~\ref{lem:computing stabilizers}. Note that each connected component $\overline Y_4$ is either equal to $\overline Y_3$, or has simply connected $\widetriangle Y_4$, and in particular, the map $\overline Y_4\to \overline Y_2$ is an embedding of a subgraph. 
\end{proof}

Combining Lemma~\ref{lem:44(2p+1)} and Lemma~\ref{lem: finitely many colored graphs} yields the following.

\begin{prop}\label{cor:442m+1}
The Artin group $G_{MNP}$ where $M = 2m+1\geq 3$ and $N= P=4$ has finite stature with respect to $\{A\}$, where $A$ is as described in Theorem~\ref{thm:splitting}. 
\end{prop}

\subsection{Case where $M,N,P\geq 3$ are all odd}\label{sec:all odd}
First consider the case where $M=N=P=3$. 

\begin{prop}\label{lem:333} 
Let $(M,N,P) =(3,3,3)$, and let $T$ be the Bass-Serre tree of the splitting $G_{333} = A*_CB$. Then for every path $\rho$ in $T$, $\Stab(\rho) = C$.\end{prop}

\begin{proof} Indeed, in this case $C$ is normal in both $A$ and $B$, so all $G_{333}$-conjugates of $C$ are equal $C$. This proves that all edge stabilizers in the action of $G_{333}$ on $T$ are equal $C$.
\end{proof}

For the remaining cases, we will apply Lemma~\ref{lem: finitely many colored graphs} to deduce that $G_{MNP}$ has finite stature with respect to $\{A\}$, similarly as in Section~\ref{sec:(2m+1,4,4)} 
We now consider the case where $M,N,P$ are all at least $5$. We continue to use Notation~\ref{notn}.
\begin{lem}\label{lem:three odd}
Let $M,N,P\geq 5$ be all odd. Every graph $\overline Y_2$ is either the left graph in Figure~\ref{fig:three odd}, or has simply connected $\widetriangle Y_2$.
Also, every graph $\overline Y_3$ is either the right graph in Figure~\ref{fig:three odd}, or has simply connected $\widetriangle Y_3$.
The map $\overline Y_4\to \overline Y_2$ is always an embedding of a subgraph.
 
\begin{figure}
\includegraphics[scale=0.4]{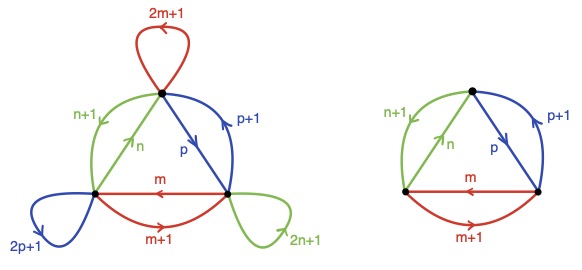}
\caption{$(M,N,P) = (2m+1, 2n+1, 2p+1)$. 
The graph on the left is the fiber product $\overline Y_2 = \overline X_C\otimes_{X_A}\overline X_C$. The graph on the right is $\sigma\beta(Y_2)\otimes_{X_A} \overline Y_2$.}\label{fig:three odd}
\end{figure}
\end{lem}
\begin{proof}
We write $M=2m+1$, $N=2n+1$, and $P=2p+1$. 
By Theorem~\ref{thm:splitting} in all the cases listed in the statement, $G_{MNP}$ splits as an amalgamated product $A*_C B$ of finite rank free groups where $[B:C] = 2$, which by Proposition~\ref{prop:artin monchrome cycles preserving} is virtually the fundamental group of a monochrome cycles preserving graph of graphs.

The first part of the lemma was proven in \cite[Lem 5.1]{JankiewiczArtinRf}. In order to prove the second part we start with computing $\sigma\beta(Y_2)$, which is illustrated in Figure~\ref{fig:three odd proof}. 
\begin{figure}
\includegraphics[scale=0.35]{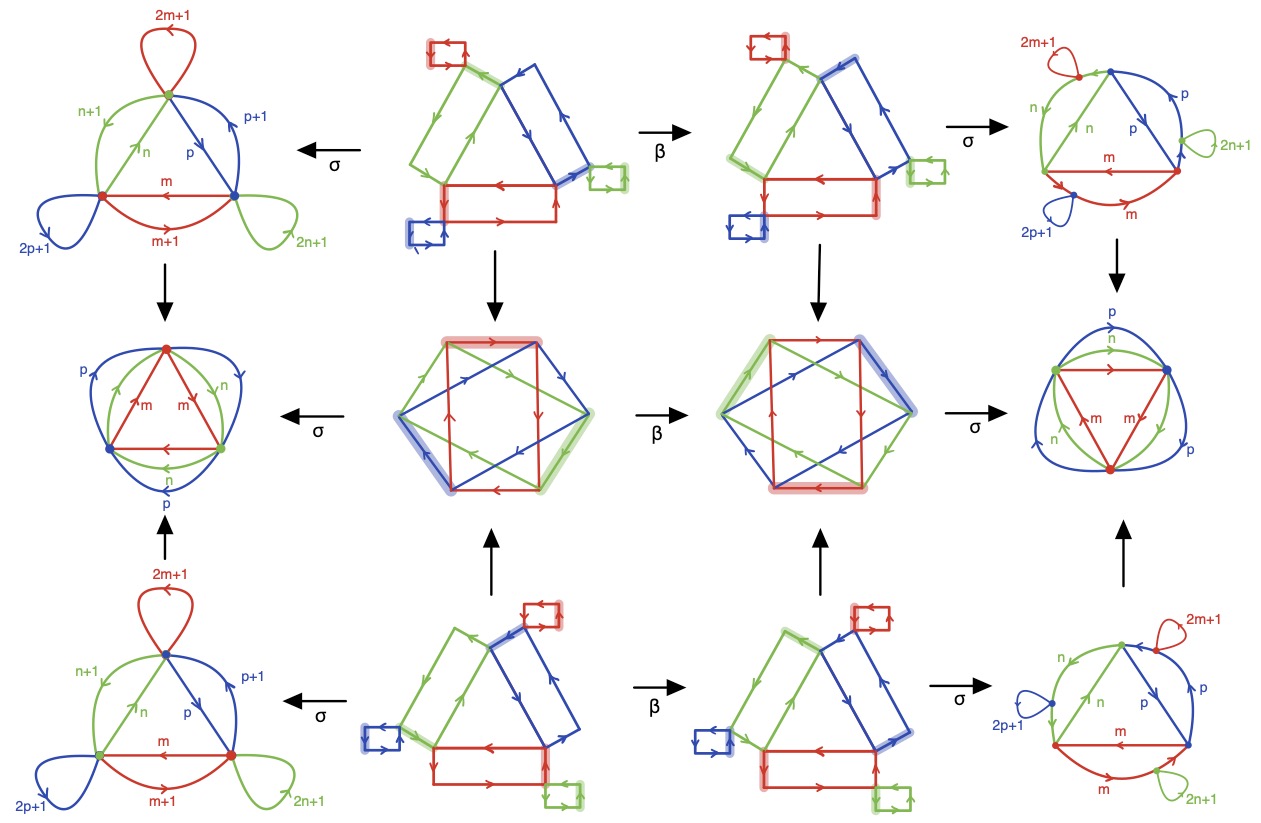}
\caption{$(M,N,P) = (2m+1, 2n+1, 2p+1)$. 
Each of the two rows of vertical arrows corresponds to respectively: $\overline Y_2\to \overline X_C$, $Y_2\to X_C$, $\beta(Y_2)\to X_C$, and $\sigma \beta(Y_2) \to \overline X_C$. 
}\label{fig:three odd proof}
\end{figure}
We note that there are two connected components $\overline Y_2$ of the fiber product $\overline X_C\otimes_{X_A}\overline X_C$ for which $\widetriangle Y_2$ is not simply connected. They are both isomorphic to the left graph in Figure~\ref{fig:three odd}, but their maps to $X_C$ are different. The first column of Figure~\ref{fig:three odd proof} shows the two combinatorial immersions $\overline Y_2\to \overline X_C$ (they are determined by the coloring of the vertices). For each $Y_2$, we compute $\sigma\beta(Y_2)$, in a similar manner as in Lemma~\ref{lem:44(2p+1)}, see the rest of Figure~\ref{fig:three odd proof}. In each case, we deduce that each connected component $\overline Y_3$ of $\overline Y_2\otimes_{X_A}\sigma\beta(Y_2)$ either has simply connected $\widetriangle Y_3$, or it is the right graph in Figure~\ref{fig:three odd}. In either case, we every map $\overline Y_4\to \overline Y_2$ is an embedding of a subgraph by a reasoning similar to one in  Lemma~\ref{lem:44(2p+1)}.
\end{proof}

We now move to the case where one or two of $M,N,P$ are equal to $3$.
Unlike in the previous case, the computation of the fiber product $\overline X_C\otimes_{X_A}\overline X_C$ in such cases was not included in \cite{JankiewiczArtinRf}. We start with that computation.
\begin{lem}\label{lem:fiberprod}
Suppose one or two of $M,N,P$ are equal to $3$. Every connected component $Y_2$ of $\overline X_C\otimes_{X_A}\overline X_C$ either has simply connected $\widetriangle Y_2$ or is 
\begin{itemize}
\item the left graph in Figure~\ref{fig:fiber product with some 3s}, when $M=3$ and $N, P\geq 5$,
\item the right graph in Figure~\ref{fig:fiber product with some 3s}, when $M=N=3$ and $P\geq 5$,
\end{itemize}
\begin{figure}
\includegraphics[scale=0.4]{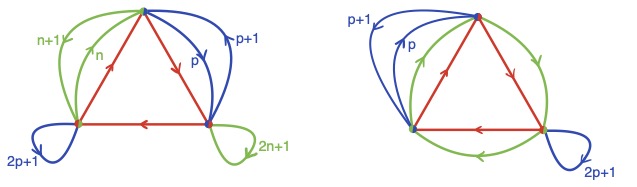}
\caption{$(M,N,P) = (2m+1, 2n+1, 2p+1)$. 
A connected component of $\overline X_C\otimes_{X_A}\overline X_C$, when $M=3$ and $N ,P\geq 5$ (left), $M=N=3$ and $P\geq 5$ (right).}\label{fig:fiber product with some 3s}
\end{figure}
\end{lem}

\begin{proof}
This is a direct computation. We remind that the graph $\overline X_C$ is the middle graph in the first row of Figure~\ref{fig:mapCtoA}. In Figure~\ref{fig:fiber product with some 3s} we bi-colored the vertices of the graphs (i.e.\ colored with a pair of colors) to make it easier for the reader to verify the computation.
\end{proof}

Now our goal is to show that $\overline Y_{\ell+2}\to \overline Y_{\ell}$ is an embedding of a subgraph for some $\ell$, so we can apply Lemma~\ref{lem: finitely many colored graphs}. The case where exactly one of $M,N,P$ is equal $3$ is considered first.

\begin{lem}\label{lem:3(2n+1)(2p+1)}
Let $N,P \geq 5$ be odd, and $M=3$. 
Every $\overline Y_3$ either has a simply connected $\widetriangle Y_3$, or is one of the graphs in  in Figure~\ref{fig:3(2n+1)(2p+1)proof2}. Moreover, the map $\overline Y_5\to \overline Y_3$ is always an embedding of a subgraph.
   \begin{figure}
\includegraphics[scale=0.7]{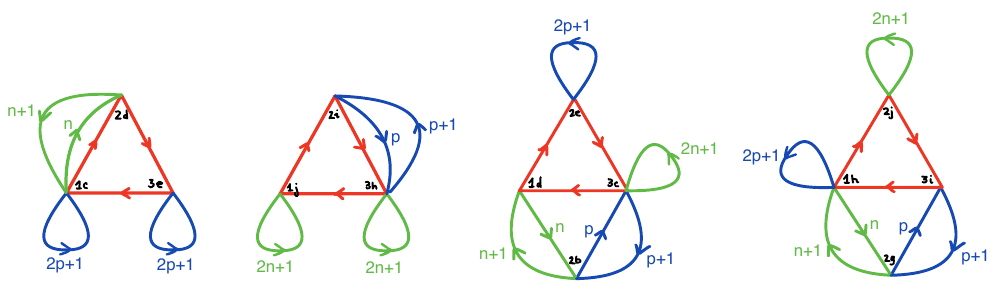}
\caption{$(M,N,P) = (3, 2n+1, 2p+1)$. All the connected components $\overline Y_3$ of $\overline Y_2\otimes_{X_A}\sigma\beta(Y_2)$ either has simply connected $\widetriangle Y_3$, or is one of the graphs pictured above. 
The labels of vertices  are $1c$, $2d$ etc, where the number corresponds to a vertex of $\overline Y_2$ and the letter corresponds to a vertex of $\sigma\beta(Y_2)$ (see Figure~\ref{fig:3(2n+1)(2p+1)proof}).}\label{fig:3(2n+1)(2p+1)proof2}
\end{figure}
 \end{lem}
 \begin{proof}
 We write $N=2n+1$ and $P=2p+1$. By Lemma~\ref{lem:fiberprod}, every $Y_2$ either has simply connected $\widetriangle Y_2$ or is the left graph in Figure~\ref{fig:fiber product with some 3s}. 
There are two components of $\overline X_C\otimes_{X_A}\overline X_C$ isomorphic to the left graph in Figure~\ref{fig:fiber product with some 3s}, each with a map to $\overline X_C$. 
For each of them we compute $\sigma\beta(Y_2)$ similarly as in Lemma~\ref{lem:three odd} and Lemma~\ref{lem:3(2n+1)(2p+1)}. This is illustrated in Figure~\ref{fig:3(2n+1)(2p+1)proof}. 
   \begin{figure}
\includegraphics[scale=0.65]{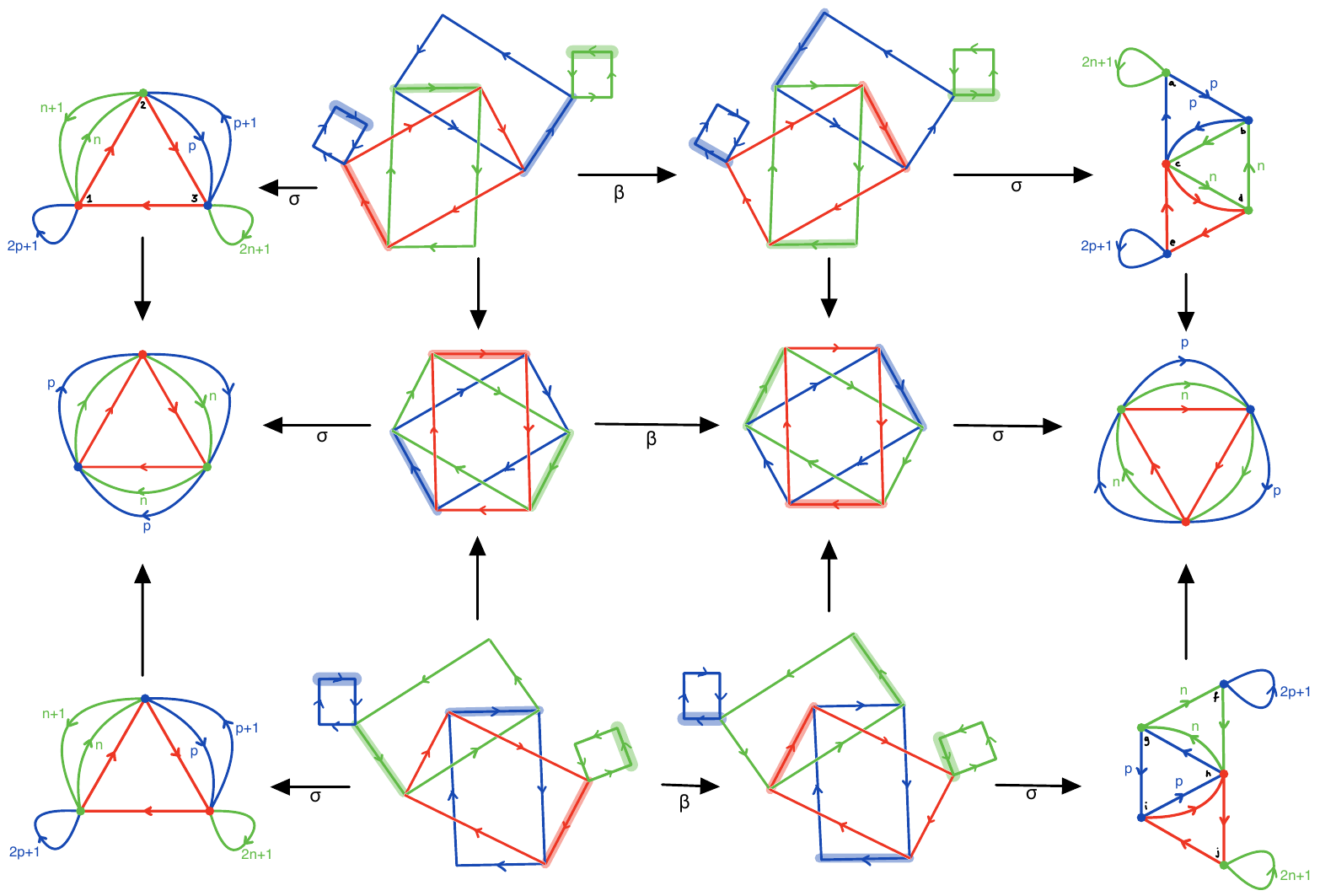}
\caption{$(M,N,P) = (3, 2n+1, 2p+1)$. Each of the two rows of vertical arrows corresponds to respectively: $\overline Y_2\to \overline X_C$, $Y_2\to X_C$, $\beta(Y_2)\to X_C$, and $\sigma \beta(Y_2) \to \overline X_C$.}\label{fig:3(2n+1)(2p+1)proof}
\end{figure}
Next, for each of the two choices of $\sigma\beta(Y_2)$ (as illustrated in Figure~\ref{fig:fiber product with some 3s}) we compute the fiber product $\overline Y_2\otimes \sigma\beta(Y_2)$, whose connected component yield $\overline Y_3$. 
The labelling of the vertices in the top left, top right and the bottom right graph in Figure~\ref{fig:fiber product with some 3s}, will help the reader to verify that every connected component $\overline Y_3$ of those fiber products are either pictured in Figure~\ref{fig:3(2n+1)(2p+1)proof2} or has simply-connected $\widetriangle Y_3$.

Finally, we compute the fiber product of the pairs of graphs from Figure~\ref{fig:3(2n+1)(2p+1)proof2}, which yield $\overline Y_4$. The only $\overline Y_4$ with non-simply-connected $\widetriangle Y_4$ is the top-left graph in Figure~\ref{fig:3(2n+1)(2p+1)proof3}, which in particular embeds in appropriate $\overline Y_3$ and is invariant under $\sigma\beta\sigma^{-1}$ (as verified in Figure~\ref{fig:3(2n+1)(2p+1)proof3}). Thus every $\overline Y_5\to \overline Y_3$ is an embedding.  

\begin{figure}
\includegraphics[scale=0.7]{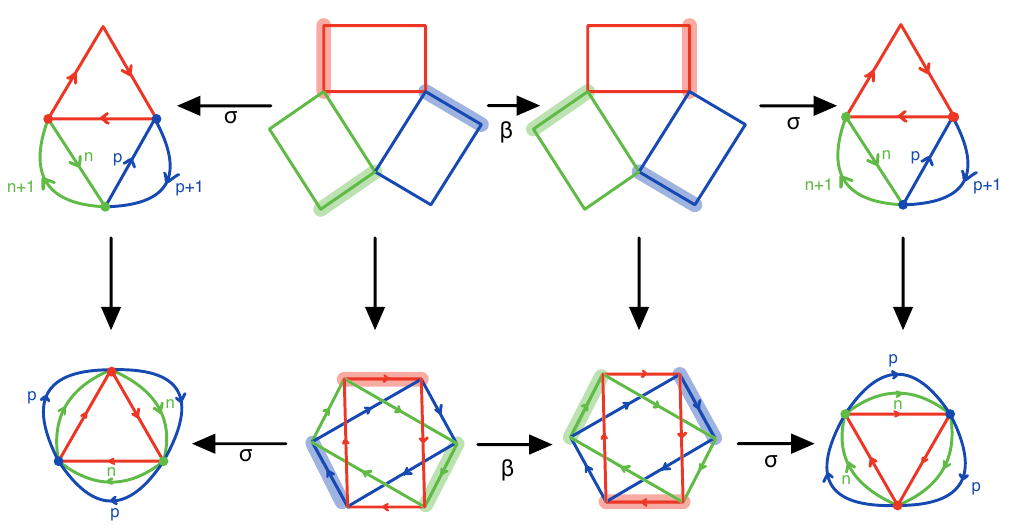}
\caption{$(M,N,P) = (3, 2n+1, 2p+1)$. The vertical arrows corresponds to respectively: $\overline Y_4\to \overline X_C$, $Y_4\to X_C$, $\beta(Y_4)\to X_C$, and $\sigma \beta(Y_4) \to \overline X_C$.}\label{fig:3(2n+1)(2p+1)proof3}
\end{figure}
%Finally, we conclude that for every $\overline Y_3$, the complex $\widetriangle Y_3$ is simply connected.
\end{proof}

In the remaining case exactly two of $M,N,P$ are equal $3$.
\begin{lem}\label{lem:33(2p+1)}
Let $P \geq 5$ be odd, and $M=N=3$. 
Every $\overline Y_3$ either has simply connected $\widetriangle Y_3$, or is one of the graphs in Figure~\ref{fig:33(2p+1)proof2}.
\begin{figure}
\includegraphics[scale=0.4]{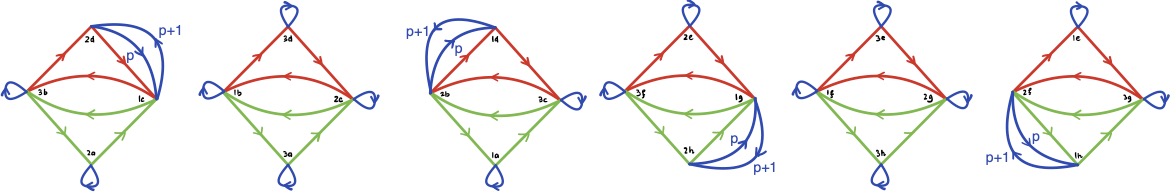}
\caption{$(M,N,P) = (3, 3, 2p+1)$. Each unlabelled blue loop has length $2p+1$.}\label{fig:33(2p+1)proof2}
\end{figure}
Moreover, the map $\overline Y_5\to \overline Y_3$ is always an embedding of a subgraph.
\end{lem}
 \begin{proof}
  We write $P=2p+1$. By Lemma~\ref{lem:fiberprod}, every $Y_2$ either has simply connected $\widetriangle Y_2$ or is isomorphic to the right graph in Figure~\ref{fig:fiber product with some 3s}. 
 
We compute $\sigma\beta(Y_2)$ similarly as in Lemma~\ref{lem:3(2n+1)(2p+1)}. This is illustrated in Figure~\ref{fig:33(2p+1)proof}.
  \begin{figure}
\includegraphics[scale=0.35]{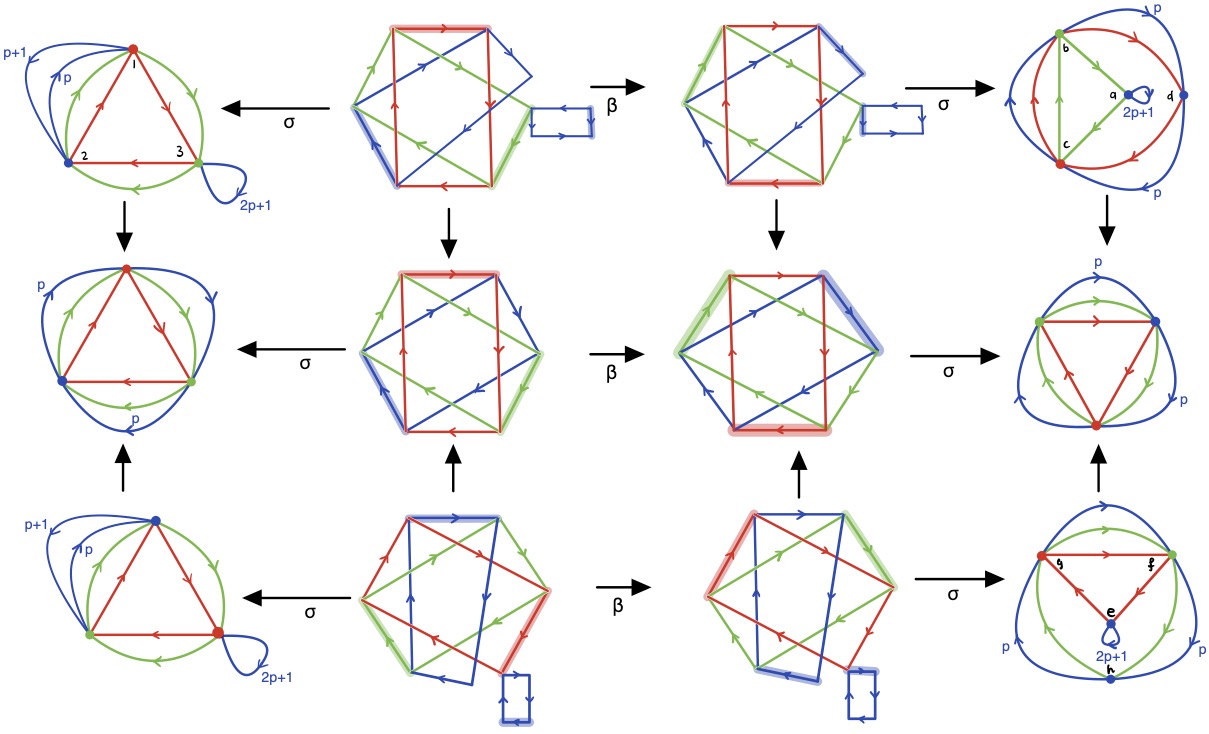}
\caption{$(M,N,P) = (3, 3, 2p+1)$. Each of the two rows of vertical arrows corresponds to respectively: $\overline Y_2\to \overline X_C$, $Y_2\to X_C$, $\beta(Y_2)\to X_C$, and $\sigma \beta(Y_2) \to \overline X_C$. }\label{fig:33(2p+1)proof}
\end{figure}
Once again, for each of the two choices of $\sigma\beta(Y_2)$ we compute the fiber product $\overline Y_2\otimes \sigma\beta(Y_2)$. As a result we obtain that $\overline Y_3$ is either a monochrome (blue) cycle, or it is isomorphic to one of the graphs in Figure~\ref{fig:33(2p+1)proof2}.

We now note that the collection of graphs in Figure~\ref{fig:33(2p+1)proof2}:
\begin{itemize}
\item has the property that the fiber product of any two graphs is a subgraph of one of the graphs in the collection, and 
\item is invariant under $\sigma\beta\sigma^{-1}$, as verified in Figure~\ref{fig:33(2p+1)proof3}.
\end{itemize}
The first fact implies that every $\overline Y_4$ is a subgraph of some $\overline Y_3$. The second fact implies that this is also the case for $\overline Y_5$.
In particular, every $\overline Y_5\to \overline Y_3$ is an embedding.
\begin{figure}
\includegraphics[scale=0.4]{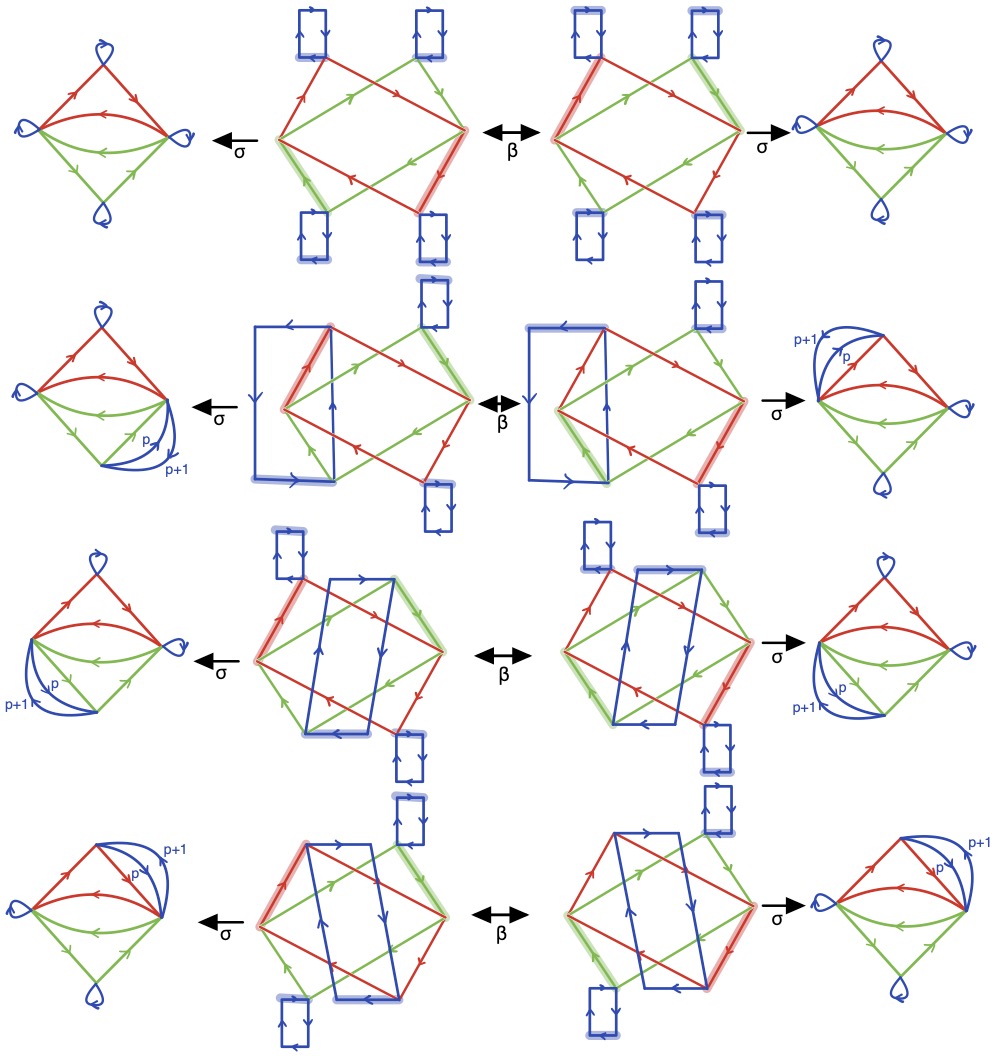}
\caption{Each unlabelled blue loop has length $2p+1$}\label{fig:33(2p+1)proof3}
\end{figure}
\end{proof}

We now summarize what we have proven in this subsection.
\begin{prop}\label{cor: all odd}
The Artin group $G_{MNP}$ where $M,N,P\geq 3$ are odd has finite stature with respect to $\{A\}$, where $A$ is as described in Theorem~\ref{thm:splitting}. 
\end{prop}

\begin{proof}
When $M=N=P=3$, the statement follows from Proposition~\ref{lem:333}. 
The case where $M=N=3$, and $P=2p+1\geq 5$ follows from Lemma~\ref{lem:33(2p+1)} and Lemma~\ref{lem: finitely many colored graphs}.
The case where $M=3$, $N=2n+1\geq 5$, and $P=2p+1\geq 5$ follows from Lemma~\ref{lem:3(2n+1)(2p+1)} and Lemma~\ref{lem: finitely many colored graphs}.
Finally, the case where $M=2m+1\geq 5$, $N=2n+1\geq 5$, and $P=2p+1\geq 5$ is a consequence of Lemma~\ref{lem:three odd} and Lemma~\ref{lem: finitely many colored graphs}.
\end{proof}
We note that the residual finiteness of $G_{333}$ follows from \cite{Squier87}. The residual finiteness of $G_{MNP}$ where $M,N,P\geq 5$ was proven in \cite{JankiewiczArtinRf}. However, the methods of \cite{JankiewiczArtinRf} do not cover the cases where one or two of $M,N,P$ are equal $3$.

\subsection{The case where $\{M,N,2\}$ where $M,N\geq 4$}\label{sec:2}
We first focus on the case where $M,N$ are both even. We recall that, unlike in the previous cases, $G_{MNP}$ splits as an HNN-extension $A*_B$, as in Theorem~\ref{thm:splitting2}.
\begin{lem}\label{lem:eveneven2 proof}
Let $M=2m, N=2n$ and $P=2$. The graphs $\overline{\phi_1X_B}$ and $\overline{\phi_2 X_B}$ are (unbased) isomorphic. In particular, the stabilizer of every finite path in the Bass-Serre tree of the splitting of $G_{MNP} = A*_B$ is conjugate to a subgroup of $A$ represented by $\overline{\phi_1X_B}$ or a wedge of monochrome cycles.
\end{lem}
\begin{proof}
The graphs  $\overline{\phi_1X_B}$ and $\overline{\phi_2 X_B}$ are computed in Theorem~\ref{thm:splitting2}, and it is easy to see that the two graphs are isomorphic. Every connected component $Y$ of the fiber product $\overline{\phi_1X_B}\otimes_{X_A} \overline{\phi_1X_B}$ is either isomorphic to $\overline{\phi_1X_B}$ or is a wedge of monochrome cycles.
\end{proof}

Next, we consider the cases where at both $M,N$ are odd.
\begin{lem}\label{lem:2oddodd}
Let $P=2$ and $M=2m+1,N=2n+1\geq 5$. Every graph $\overline Y_2$ either is isomorphic  to the left graph in Figure~\ref{fig:2andsomeodd}(a) or it is a wedge of monochrome cycles. If $Y$ is the left graph in Figure~\ref{fig:2andsomeodd}(a), then $\sigma\beta(Y)$ is (unbased) isometric to $Y$. Therefore, every graph $\overline Y_i$ either one of the two graphs in Figure~\ref{fig:2andsomeodd}(a), or it is a wedge of monochrome cycles. 
\end{lem}
\begin{figure}
\includegraphics[scale=0.4]{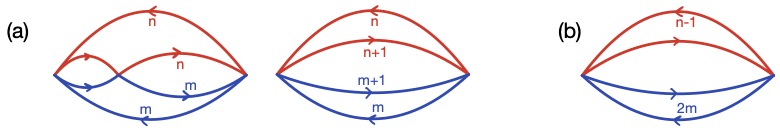}
\caption{$P = 2$. All the graphs $\overline Y_\ell$ are either wedges of circles, or one of the graphs above, when (a) $M,N\geq 5$ are both odd and $P=2$, and (b) exactly one of $M,N\geq 4$ is odd and $P=2$. }\label{fig:2andsomeodd}
\end{figure}
\begin{proof} The first statement was proven in \cite[Rem 3.5]{JankiewiczArtinSplittings}. The proof of the second statement is illustrated in Figure~\ref{fig:invarianceofY2}(a). Let $\overline Y_2$ be the left graph in Figure~\ref{fig:2andsomeodd}(a). Then  every connected component $\overline Y_3$ of the fiber product $\overline{Y}_2\otimes_{X_A} \sigma\beta(Y_2) = \overline{Y}_2\otimes_{X_A} \overline{Y}_2$ is a wedge of monochrome cycles, is isomorphic to $\overline Y_2$ or to the right graph in Figure~\ref{fig:2andsomeodd}(a). 
We also note that if $Y$ is the right graph in Figure~\ref{fig:2andsomeodd}(b), then $\sigma\beta(Y)$ is isometric to $Y$. We conclude that every graph $\overline Y_\ell$ either one of the two graphs in Figure~\ref{fig:2andsomeodd}(a), or it is a wedge of monochrome cycles. 
\begin{figure}
\includegraphics[scale=0.4]{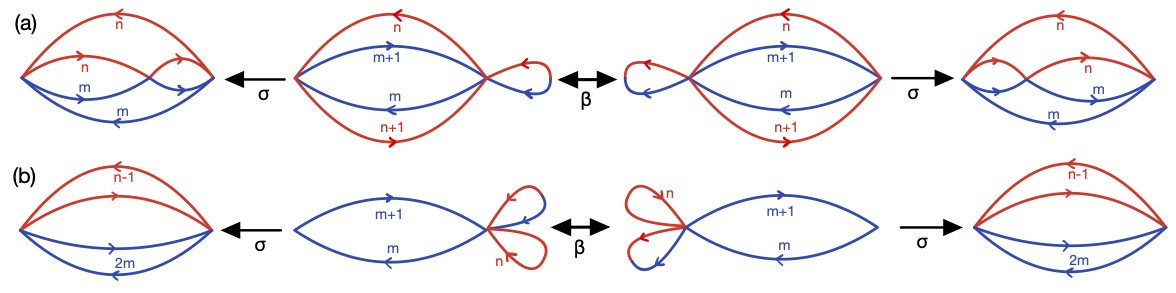}
\caption{$(M,N,2) = (2m+1, N, 2)$. In case (a) $N=2n+1$, and in case (b) $N=2n$. If $\overline Y$ is the rightmost graphs, then it is isometric to $\sigma\beta(Y)$. }\label{fig:invarianceofY2}
\end{figure}
\end{proof}

Finally, we consider the cases where exactly one of $M,N$ is odd.

\begin{lem}\label{lem:2oddeven}
Let $P=2$, $M=2m+1\geq 5$, and $N=2n\geq 4$. Every graph $\overline Y_2$ either is isometric to the graph in Figure~\ref{fig:2andsomeodd}(b) or it is a wedge of monochrome cycles. If $Y$ is the graph in Figure~\ref{fig:2andsomeodd}(b), then $\sigma\beta(Y)$ is (unbased) isometric to $Y$. Therefore, every graph $\overline Y_i$ either one of the graph in Figure~\ref{fig:2andsomeodd}(b), or it is a wedge of monochrome cycles. 
\end{lem}
\begin{proof} The first statement was proven in \cite[Prop 3.4]{JankiewiczArtinSplittings}. The proof of the second statement is illustrated in Figure~\ref{fig:invarianceofY2}(b). Let $Y$ denote the graph in Figure~\ref{fig:invarianceofY2}(b). Every connected component of the fiber product $Y\otimes_{X_A} Y$ is either isometric to $Y$ or it is a wedge of monochrome cycles.
\end{proof}

\begin{prop}\label{cor: 2MN}
The Artin group $G_{MN2}$ where $M,N\geq 4$ has finite stature with respect to $\{A\}$, where $A$ is as described in Theorem~\ref{thm:splitting}.
\end{prop}

\begin{proof}All the cases can be deduced from Lemma~\ref{lem: finitely many colored graphs} together with 
\begin{itemize}
\item Lemma~\ref{lem:eveneven2 proof} when $M,N$ are both even;
\item Lemma~\ref{lem:2oddeven} when exactly one of $M,N$ is odd;
\item Lemma~\ref{lem:2oddodd} when both $M,N$ are odd.\qedhere
\end{itemize}
\end{proof}

Residual finiteness of $G_{MN2}$ where at least one of $M,N$ is even was proven in \cite{JankiewiczArtinSplittings}, but the case of both $M,N$ odd is a new result.

\subsection{Triangle Artin groups with label $\infty$} Note that all  of the above proofs are valid if any of the labels $M,N,P$ are equal to $\infty$.

\bibliographystyle{alpha}
\bibliography{../../../kasia}

\end{document}